\documentclass[a4paper]{article}
\usepackage{amsmath,amssymb,amsfonts,amsthm} 
\usepackage{mathrsfs} 
\usepackage{enumerate} 
\usepackage{cases} 
\usepackage{float} 
\usepackage{graphicx} 
\allowdisplaybreaks[4] 
\usepackage{wrapfig}  
\usepackage{url} 
\usepackage[usenames]{color} 
\usepackage{setspace} 

\def\b1{\mbox{\boldmath $1$}} 

\makeatletter
\newcommand{\Biggg}{\bBigg@{3.5}} 
\makeatother

\theoremstyle{plain} 
\newtheorem{theorem}{\bf Theorem}[section] 
\newtheorem{corollary}{\bf Corollary}[section] 
\newtheorem{proposition}{\bf Proposition}[section] 
\newtheorem{definition}{\bf Definition}[section] 
\newtheorem{remark}{\bf Remark}[section]  

\makeatother
\allowdisplaybreaks 

\usepackage[numbers,sort&compress]{natbib} 
\usepackage[bookmarksnumbered,bookmarksopen,colorlinks,citecolor=blue,linkcolor=blue]{hyperref} 

\begin{document}
\date{}

\title{Turnpike properties for zero-sum stochastic linear quadratic differential games of Markovian regime switching system}
\author{Xun Li \thanks{Department of Applied Mathematics, The Hong Kong Polytechnic University, Hong Kong, China; E-mail: li.xun@polyu.edu.hk}\qquad
Fan Wu \thanks{Corresponding Author: School of Big Data and Statistics, Anhui University, Hefei 230601, China; E-mail:wfyy121107@163.com}\qquad
Xin Zhang \thanks{School of Mathematics, Southeast University, Nanjing 211189, China; E-mail: x.zhang.seu@gmail.com}}

\maketitle

\noindent{\bf Abstract:} This paper investigates the long-time behavior of zero-sum stochastic linear–quadratic (SLQ) differential games within Markov regime-switching diffusion systems and establishes the turnpike property of the optimal triple. By verifying the convergence of the associated coupled differential Riccati equations (CDREs) along with their convergence rate, we show that, for a sufficiently large time horizon, the equilibrium strategy in the finite-horizon problem can be closely approximated by that of the infinite-horizon problem. Furthermore, this study enhances and extends existing results concerning zero-sum SLQ differential games over both finite and infinite horizons.

\medskip

\noindent{\bf Keywords:} Zero-sum stochastic differential games;  Linear quadratic; Turnpike property; Coupled differential Riccati equations; Markov regime-switching diffusion system.

\section{Introduction}\label{section-1}
Let $(\Omega,\mathcal{F},\mathbb{P} )$ be a complete probability space with the natural filtration $\mathbb{F} :=
\{ \mathcal{F}_t\}_{t\geq 0}$ generated by a standard one-dimensional Brownian motion $W=\{W(t)\}_{t\geq 0}$ and a continuous time irreducible Markov chain $\alpha=\{\alpha_t\}_{t\geq 0}$ with a finite state space $\mathcal{S}:=\left\{1,2,\cdots,L\right\}$.
We let $\mathbb{R}^{n\times m}$ denote the Euclidean space of all $ n \times m$  matrices and set $\mathbb{R}^{n}:=\mathbb{R}^{n\times 1}$ for simplicity. In addition, the set of all $ n\times n$ symmetric matrices is denoted by $\mathbb{S}^n$. Specially, the sets of all $ n\times n$ semi-positive definite matrices and positive definite matrices are denoted by $\overline{\mathbb{S}_{+}^n}$ and $\mathbb{S}_{+}^n$, respectively. Further, for any $M, N \in \mathbb{S}^n$, we write $M \geqslant N$ (respectively, $M>N$) if $M-N$ is semi-positive definite (respectively, positive definite). Let $\mathcal{P}$ be the $\mathbb{F}$ predictable $\sigma$-field on $[0,\infty)\times\Omega$ and we write $\varphi \in \mathcal{P}$  (respectively, $\varphi \in \mathbb{F}$) if it is $\mathcal{P}$-measurable (respectively, $\mathbb{F}$-progressively measurable). Then, for any Euclidean space $\mathbb{H}$ and a time interval $\Gamma\subseteq [0,\infty)$, we introduce the following spaces:
$$
\begin{array}{l}
C(\Gamma; \mathbb{H}) =\Big\{\varphi: \Gamma  \rightarrow \mathbb{H}  \mid \varphi(\cdot) \text{ is a continuous function } \Big\},\\[3mm]
L_{\mathbb{G}}^{2}(\Gamma; \mathbb{H}) =\Big\{\varphi: \Gamma \times \Omega \rightarrow \mathbb{H} \mid \varphi(\cdot) \in \mathbb{G}\text{, } \mathbb{E} \int_{\Gamma}|\varphi(t)|^{2} dt<\infty\Big\},\quad \mathbb{G}=\mathbb{F}, \ \mathcal{P},\\[3mm]
L_{\mathbb{F}}^{2,loc}(\mathbb{H}) =\Big\{\varphi: [0, \infty)\times \Omega \rightarrow \mathbb{H} \mid \varphi(\cdot) \in \mathbb{F}\text{, } \mathbb{E} \int_{0}^{T}|\varphi(s)|^{2} ds<\infty, \ \forall T>0\Big\}.
\end{array}
$$
For simplicity, we denote $L_{\mathbb{F}}^{2}(\mathbb{H})=L_{\mathbb{F}}^{2}\left([0,\infty);\mathbb{H}\right)$ and $L_{\mathcal{P}}^{2}(\mathbb{H})=L_{\mathcal{P}}^{2}\left([0,\infty);\mathbb{H}\right)$.

Based on the above setting, we consider the following controlled  linear  stochastic differential equation (SDE):
 \begin{equation}\label{ZLQ-state}
   \left\{
   \begin{aligned}
   dX(t)&=\left[A(\alpha_{t})X(t)+B_{1}(\alpha_{t})u_{1}(t)+B_{2}(\alpha_{t})u_{2}(t)\right]dt\\
   &\quad+\left[C(\alpha_{t})X(t)+D_{1}(\alpha_{t})u_{1}(t)+D_{2}(\alpha_{t})u_{2}(t)\right]dW(t),\qquad t\geq0,\\
   X(0)&=x,\quad \alpha(0)=i,
   \end{aligned}
   \right.
 \end{equation}
 and performance functional with time horizon $T>0$:
 \begin{equation}\label{ZLQ-performance}
\begin{aligned}
    J_{T}\left(x,i;u_{1}(\cdot),u_{2}(\cdot)\right)
    \triangleq \mathbb{E}\int_{0}^{T}
    \left<
    \left(
    \begin{matrix}
    Q(\alpha_{t}) & S_{1}(\alpha_{t})^{\top} & S_{2}(\alpha_{t})^{\top} \\
    S_{1}(\alpha_{t}) & R_{11}(\alpha_{t}) & R_{12}(\alpha_{t}) \\
    S_{2}(\alpha_{t}) & R_{21}(\alpha_{t}) & R_{22}(\alpha_{t})
    \end{matrix}
    \right)
    \left(
    \begin{matrix}
    X(t) \\
    u_{1}(t) \\
    u_{2}(t)
    \end{matrix}
    \right),
    \left(
    \begin{matrix}
    X(t) \\
    u_{1}(t) \\
    u_{2}(t)
    \end{matrix}
    \right)
    \right>dt.
  \end{aligned}
\end{equation}
 In the above, $X(\cdot)\triangleq X(\cdot;x,i,u_{1},u_{2})\in \mathbb{R}^{n}$, is called the \emph{state process}, and $u_{k}(\cdot)\in \mathbb{R}^{m_{k}}$, is called the \emph{control process} of player $k$. Additionally, for $i\in\mathcal{S}$, the coefficients in state equation \eqref{ZLQ-state} and performance functional \eqref{ZLQ-performance} satisfies:
 $$
 A(i),\,C(i)\in\mathbb{R}^{n\times n},\quad B_{k}(i),\, D_{k}(i)\in\mathbb{R}^{n\times m_{k}},\quad k=1,2,
 $$
 and
 $$
 Q(i) \in \mathbb{S}^{n},   \quad
 R_{kk}(i) \in \mathbb{S}^{m_{k}},   \quad
 S_{k}(i)\in\mathbb{R}^{m_{k} \times n},   \quad
 R_{12}(i)=R_{21}(i)^{\top}\in\mathbb{R}^{m_{1} \times m_{2}},\quad k=1,2.
 $$
 For $k=1,2$, let $\mathcal{U}_{k}[0,T]=L_{\mathbb{F}}^{2}([0,T];\mathbb{R}^{m_{k}})$ and $\mathcal{U}[0,T]=\mathcal{U}_{1}[0,T]\times \mathcal{U}_{2}[0,T]$. 
 Thus, the zero-sum stochastic linear quadratic (SLQ) differential game with regime switching over the time horizon $[0,T],\,\,(T>0)$ can be summarized as follows.

\textbf{Problem (M-ZLQ)$_{T}$.} For any $(x,i)\in \mathbb{R}^{n}\times \mathcal{S}$, find a $(\bar{u}_{1,T}(\cdot),\bar{u}_{2,T}(\cdot))\in\mathcal{U}_{ad}(x,i)$ such that
\begin{equation}\label{ZLQ-problem}
\begin{aligned}
J\left(x,i;\bar{u}_{1,T}(\cdot),u_{2}(\cdot)\right)\leq J\left(x,i;\bar{u}_{1,T}(\cdot),\bar{u}_{2,T}(\cdot)\right)\triangleq V_{T}(x,i)\leq J\left(x,i;u_{1}(\cdot),\bar{u}_{2,T}(\cdot)\right),\\
\forall (u_{1}(\cdot),u_{2}(\cdot))\in \mathcal{U}[0,T].
\end{aligned}
\end{equation}

From above, we see that Player 1 aims to minimize \eqref{ZLQ-problem} through choosing a control $u_{1}(\cdot)$, while Player 2 aims to maximize \eqref{ZLQ-problem} by selecting a control $u_{2}(\cdot)$. Hence, the performance functional \eqref{ZLQ-problem} represents the \emph{cost} for Player 1 and the \emph{payoff} for Player 2. The pair $(\bar{u}_{1,T}(\cdot),\bar{u}_{2,T}(\cdot))$ (if it exist) is called an \emph{open-loop saddle strategy} of Problem (M-ZLQ)$_{T}$ at the initial pair $(x,i)$, $\bar{X}_{T}(\cdot)$ is called the corresponding \emph{open-loop optimal state process}, and $V_{T}(\cdot,\cdot)$ is called the \emph{value function} of the game. We also refer to $(\bar{X}_{T}(\cdot),\bar{u}_{1,T}(\cdot),\bar{u}_{2,T}(\cdot))$ as an \emph{open-loop optimal triple} at $(x,i)$. When such a triple exists for every initial pair $(x,i)$, we say that the Problem (M-ZLQ)$_{T}$ is \emph{open-loop solvable}.

It is worth to mention that, for any given initial pair $(x,i)$ and $(u_{1}(\cdot),u_{2}(\cdot))\in \mathcal{U}_{1}[0,\infty)\times \mathcal{U}_{2}[0,\infty)$, the state process $X(\cdot;x,i,u_{1},u_{2})$ defined in \eqref{ZLQ-state} is typically in $L_{\mathbb{F}}^{2,loc}(\mathbb{R}^{n})$, which can not ensure  the well-
posedness of the performance functional $J_{\infty}(x,i;u_{1}(\cdot),u_{2}(\cdot))$.  In the following, we call a control pair
$(u_{1}(\cdot),u_{2}(\cdot))\in \mathcal{U}_{1}[0,\infty)\times \mathcal{U}_{2}[0,\infty)$ admissible for the initial pair $(x,i)$ of Problem (M-ZLQ)$_{\infty}$ if the corresponding state process $X(\cdot;x,i,u_{1},u_{2})\in L_{\mathbb{F}}^{2}(\mathbb{R}^{n})$, and denote the set of admissible control pairs for the initial pair $(x,i)$ by $\mathcal{U}_{ad}(x,i)$. Clearly, the performance functional $J_{\infty}(x,i;u_{1}(\cdot),u_{2}(\cdot))$ is well-defined for any control pair $(u_{1}(\cdot),u_{2}(\cdot))\in \mathcal{U}_{ad}(x,i)$. In general, $\mathcal{U}_{ad}(x,i)$ depends on the initial pair $(x,i)$ and is only a subset of $\mathcal{U}_{1}[0,\infty)\times \mathcal{U}_{2}[0,\infty)$. However, as we can see in the next section, we can obtain
$\mathcal{U}_{ad}(x,i)=\mathcal{U}_{1}[0,\infty)\times \mathcal{U}_{2}[0,\infty)$ under appropriate condition.

In this paper, we are going to investigate the long-time asymptotic behavior of Problem (M-ZLQ)$_{T}$ as $T\rightarrow \infty$. It turns out that under appropriate conditions, the optimal triple of Problem (M-ZLQ)$_{T}$ exhibits the so-called \emph{exponential turnpike property}. Specifically, let $(\bar{X}_{T}(\cdot),\bar{u}_{1,T}(\cdot),\bar{u}_{2,T}(\cdot))$ be the optimal triple of Problem (M-ZLQ)$_{T}$. The exponential turnpike property asserts the existence of some constants $K,\,\mu>0$ (independent of the initial pair $(x,i)$ and the terminal horizon $T$) such that
\begin{equation}\label{exponential-turnpike-property-introduction}
\mathbb{E}\left[
\left|\bar{X}_{T}(t)-\bar{X}_{\infty}(t)\right|^{2}+\left|\bar{u}_{1,T}(t)-\bar{u}_{1,\infty}(t)\right|^{2}+\left|\bar{u}_{2,T}(t)-\bar{u}_{2,\infty}(t)\right|^{2}
\right]\leq K\left|x\right|^{2}\left[e^{-\mu(T-t)}+e^{-\mu t}\right],
\end{equation}
where $(\bar{X}_{\infty}(\cdot),\bar{u}_{1,\infty}(\cdot),\bar{u}_{2,\infty}(\cdot))$ is the corresponding optimal triple of Problem (M-ZLQ)$_{\infty}$.

Let $\kappa\in(0,\frac{1}{2})$ be an arbitrary number. Then inequality \eqref{exponential-turnpike-property-introduction} implies that 
$$
\mathbb{E}\left[
\left|\bar{X}_{T}(t)-\bar{X}_{\infty}(t)\right|^{2}+\left|\bar{u}_{1,T}(t)-\bar{u}_{1,\infty}(t)\right|^{2}+\left|\bar{u}_{2,T}(t)-\bar{u}_{2,\infty}(t)\right|^{2}
\right]\leq 2K\left|x\right|^{2} e^{-\mu\kappa T},\quad \forall t\in [\kappa T, (1-\kappa) T)].
$$
Since $K$ and $\mu$ are independent of $T$, the optimal triple $(\bar{X}_{T}(\cdot),\bar{u}_{1,T}(\cdot),\bar{u}_{2,T}(\cdot))$ remains very close to the optimal triple $(\bar{X}_{\infty}(\cdot),\bar{u}_{1,\infty}(\cdot),\bar{u}_{2,\infty}(\cdot))$ when the time horizon $[0,T]$ is very large. Therefore, we can use $(\bar{X}_{\infty}(\cdot),\bar{u}_{1,\infty}(\cdot),\bar{u}_{2,\infty}(\cdot))$ to approximately solve Problem (M-ZLQ)$_{T}$, and inequality \eqref{exponential-turnpike-property-introduction} provides an error estimate for this approximate solution.

The study of two-person zero-sum stochastic differential games can be traced back to Fleming and Souganidis \cite{fleming1989existence}. Later, Hamad\'ene and Lepeltier \cite{hamadene1995zero} explored these games using the backward stochastic differential equation (BSDE) approach. Following this, several subsequent works appeared afterward. See, for examples, \citet{buckdahn2008stochastic,wang2010pontryagin,wang2012partial,bayraktar2013weak,lv2020two}, and so on. Recently, the zero-sum SLQ differential game has garnered significant research interest. Mou et al \cite{mou2006two} examined the open-loop solvability of the zero-sum SLQ differential game problem (referred to as Problem (ZLQ)) by employing the Hilbert space approach. Sun and Yong \cite{Sun.J.R.2014_NILQ} extended this result by investigating the closed-loop solvability of Problem (ZLQ). However, both \citet{mou2006two} and \citet{Sun.J.R.2014_NILQ} did not address the conditions under which the associated differential Riccati equation (DRE) admits a solution. Yu \cite{yu2015optimal} studied the optimal feedback control for Problem (ZLQ) using the Riccati equation approach and proved the solvability of the DRE in a specific case. \citet{Sun2021} revisited Problem (ZLQ) and established the solvability of the DRE under the uniform convexity-concavity assumption. Inspired by these contributions, our previous works \cite{My-paper-2024-finite-M-ZLQ, My-paper-2024-infinite-M-ZLQ} investigated the zero-sum SLQ control problem over finite and infinite horizons, respectively.

Regarding the turnpike phenomenon, early investigations can be traced to the seminal contributions of \citet{ramsey1928mathematical} and \citet{neumann1945model}. The term ``turnpike" was coined by \citet{dorfman1987linear} in 1958, drawing an analogy to the typical features of American toll highways. Recently, significant progress has been achieved in establishing the turnpike property for deterministic optimal control problems; see \cite{porretta2013long,damm2014exponential,grune2018turnpike,lou2019turnpike,breiten2020turnpike,sakamoto2021turnpike} and the references therein. Nevertheless, investigations into the turnpike property for stochastic control problems remain relatively limited. To our knowledge, \citet{sun2022turnpike} was the first to examine the turnpike property for SLQ control problems,  followed by subsequent studies \cite{sun2024turnpike,sun2024turnpike-2,trelat2025exponential,schiessl2025turnpike}. Recently, \citet{sun2024long} further investigated the long-time behavior of zero-sum SLQ differential games and established the associated turnpike property.

In this paper, we investigate the turnpike property of zero-sum SLQ differential games for Markov regime-switching diffusion systems, building on our previous works \cite{My-paper-2024-finite-M-ZLQ,My-paper-2024-infinite-M-ZLQ}. The main contributions and challenges are outlined as follows:
\begin{enumerate}
  \item Establishing the turnpike property for Problem (M-ZLQ)$_{T}$ requires a thorough analysis of the convergence and convergence rate of solutions to the associated CDREs. Unlike the SLQ control problem for Markov regime-switching diffusion systems studied in \citet{mei2025turnpike}, solutions to these CDREs arising from zero-sum SLQ differential games lack monotonicity. Moreover, analyzing convergence and its rate for this coupled system is considerably more challenging than for a single differential Riccati equation, as seen in zero-sum SLQ differential games driven by diffusion models. Therefore, our results extend those reported in \citet{mei2025turnpike} and \citet{sun2024long}.

  \item The turnpike property \eqref{exponential-turnpike-property-introduction} for zero-sum SLQ differential games is established, under which one can approximate the optimal triple of Problem (M-ZLQ)$_{T}$ using the optimal open-loop saddle strategy of Problem (M-ZLQ)$_{\infty}$ when the time horizon $T$ is large enough. Notably, the closed-loop representation of the optimal open-loop saddle strategy for Problem (M-ZLQ)$_{\infty}$ depends solely on a set of coupled algebraic Riccati equations (CAREs), which are independent of the terminal time $T$, thereby offering substantial computational and practical advantages.

  \item Although our previous works \cite{My-paper-2024-finite-M-ZLQ, My-paper-2024-infinite-M-ZLQ} have, to some extent, addressed the zero-sum SLQ differential game over both finite and infinite horizons, several issues remain open and deserve further investigation (see Remark \ref{rmk-question}). In the process of establishing the turnpike property for Problem (M-ZLQ)$_{T}$, this study refines and extends the results previously reported in  \cite{My-paper-2024-finite-M-ZLQ, My-paper-2024-infinite-M-ZLQ}.
\end{enumerate}

The rest of the paper is organized as follows. Section \ref{section-2} introduces some basic notations and preliminary results for analyzing Problem (M-ZLQ)$_{T}$ and Problem (M-ZLQ)$_{\infty}$, whose unique solvability is provided in Section \ref{section-3}. Section \ref{section-4} shows the convergence of optimal triples between Problem (M-ZLQ)$_{T}$ and Problem (M-ZLQ)$_{\infty}$ in a suitable sense. Section \ref{section-5} further investigates this convergence rate and establishes the turnpike property.

\section{Preliminaries}\label{section-2}
We begin this section by introducing additional useful notations beyond those in the previous section. Let $\Pi:=[\pi_{ij}]_{i,j=1,\cdots,L}$ be the generator of the Markov chain $\alpha$ and $N_{j}(t)$ be the number of jumps into state $j$ up to time $t$ and set
$$\widetilde{N}_{j}(t)\triangleq N_{j}(t)-\int_{0}^{t}\lambda_{j}(s)ds,\quad \text{with}\quad \lambda_{j}(s)\triangleq \sum_{i\neq j}^{L}\pi_{ij}\mathbb{I}_{\{\alpha_{s-}=i\}}.$$
Then, for each $j\in \mathcal{S}$, the process $\widetilde{N}_{j}(\cdot)$ is an $\left(\mathbb{F},\mathbb{P}\right)$-martingale. For any given L-dimensional vector process $\mathbf{\Gamma}(\cdot)=\left[\Gamma_1(\cdot),\Gamma_2(\cdot),\cdots,\Gamma_L(\cdot)\right]$, we define
$$\mathbf{\Gamma}(s)\cdot d\mathbf{\widetilde{N}}(s)\triangleq\sum_{j=1}^{L}\Gamma_{j}(s)d\widetilde{N}_{j}(s). $$
We further let $M^{\top}$ denote the transpose of a matrix  $M$ and $\langle\cdot, \cdot\rangle$ denote the inner products in possibly different Hilbert spaces.

For any Banach space $\mathbb{B}$, we denote
$$\mathcal{D}\left(\mathbb{B}\right)\triangleq\left\{\mathbf{\Lambda}=\left(\Lambda(1),\cdots,\Lambda(L)\right) \mid \Lambda(i) \in \mathbb{B}\text{, } \forall i\in \mathcal{S}\right\}.$$
If \(\mathbf{\Lambda} \in \mathcal{D}(\mathbb{R}^{n \times n})\), we define \(\lambda\) and \(\mu\) as the smallest and largest eigenvalues of \(\mathbf{\Lambda}\), respectively, i.e.,
\[
\lambda = \min_{i \in \mathcal{S}} \lambda_i,
\quad
\mu = \max_{i \in \mathcal{S}} \mu_i,
\]
where $\lambda_{i}$ and $\mu_{i}$ are the smallest and largest eigenvalue of $\Lambda(i)$, $i\in\mathcal{S}$, respectively. Without causing confusion, we sometimes also say that $\lambda$ and $\mu$ are the smallest and largest eigenvalues of process $\Lambda(\alpha)$.

In what follows, we denote $m=m_{1}+m_{2}$, $u(\cdot)=(u_{1}(\cdot)^{\top},u_{2}(\cdot)^{\top})^{\top}$ and 
\begin{equation}\label{notation}
\begin{aligned}
&B(i)=(B_{1}(i),B_{2}(i)),\  D(i)=(D_{1}(i),D_{2}(i)), \
S(i)=\left(\begin{array}{c}
S_{1}(i)\\S_{2}(i)
\end{array}\right),\
R(i)=\left(\begin{array}{cc}R_{11}(i) &R_{12}(i)\\R_{21}(i) &R_{22}(i)
\end{array}\right), \ i\in \mathcal{S}.
\end{aligned}
\end{equation}
Thus, the state system can be rewritten as
\begin{equation}\label{state-2}
\left\{
\begin{array}{l}
dX(t)=\left[A(\alpha_{t})X(t)+B(\alpha_{t})u(t)\right]dt+\left[C(\alpha_{t})X(t)+D(\alpha_{t})u(t)\right]dW(t),\quad t\geq 0,\\
X(0)=x,\quad \alpha_{0}=i.
\end{array}
\right.
\end{equation}
For convenience, we denote the state system \eqref{state-2} as $[A,C;B,D]_{\alpha}$ and as $[A,C]_{\alpha}$ for $B=D\triangleq 0$, i.e., $[A,C]_{\alpha}\triangleq [A,C;0,0]_{\alpha}$.  We now introduce the following definitions of stability.
\begin{definition}
  \begin{description}
    \item[(i)] System $[A,C]_{\alpha}$ is said to be $L^{2}$-stable if its solution $X(\cdot;x,i)$ with initial pair $(x,i)$ satisfies
\begin{equation}\label{L2-stable}
 \mathbb{E}\int_{0}^{\infty}\left|X(t;x,i)\right|^{2}dt<\infty,\quad \forall (x,i)\in\mathbb{R}^{n}\times\mathcal{S}.
\end{equation}
    \item[(ii)] System $[A,C;B,D]_{\alpha}$  is said to be $L^{2}$-stabilizable if there exist an element $\mathbf{\Theta}\in\mathcal{D}(\mathbb{R}^{m\times n})$ such that the following system is $L^{2}$-stable:
\begin{equation}\label{state-closed}
\left\{
\begin{array}{l}
dX(t)=\left[A(\alpha_{t})+B(\alpha_{t})\Theta(\alpha_t)\right]X(t)dt+\left[C(\alpha_{t})+D(\alpha_{t})\Theta(\alpha_t)\right]X(t)dW(t),\quad t\geq 0,\\
X(0)=x,\quad \alpha_{0}=i,
\end{array}
\right.
\end{equation}
We denote the element $\mathbf{\Theta} \in \mathcal{D}(\mathbb{R}^{m \times n})$ as a stabilizer of the system $[A,C;B,D]_{\alpha}$, and the set of all such stabilizers by $\mathcal{H}[A,C;B,D]_{\alpha}$.
  \end{description}
\end{definition}

Suppose the system $[A,C]_{\alpha}$ is $L^{2}$-stable. Then, by Proposition 2.5 in Wu et al. \cite{My-paper-2025-infinite-M-SLQ}, the following system
\begin{equation}\label{AC-b-sigma}
\left\{
\begin{array}{l}
dX(t) = [A(\alpha_t) X(t) + b(t)]\,dt + [C(\alpha_t) X(t) + \sigma(t)]\,dW(t), \quad t \geq 0,\\
X(0)=x,\quad \alpha_0 = i,
\end{array}
\right.
\end{equation}
admits a unique solution $X(\cdot; x,i)\in L_{\mathbb{F}}^{2}(\mathbb{R}^{n})$ for every initial pair $(x,i)$ and any $b(\cdot),\, \sigma(\cdot)\in L_{\mathbb{F}}^{2}(\mathbb{R}^{n})$. Consequently, the admissible control set $\mathcal{U}_{ad}(x,i)$ of Problem (M-ZLQ)$_{\infty}$ can be represented as
$$
\mathcal{U}_{ad}(x,i) = \mathcal{U}_1[0,\infty) \times \mathcal{U}_2[0,\infty).
$$
For simplicity, we assume the following throughout this paper.

\textbf{(A1).} There exists a constant $\delta > 0$ such that for every $T \in (0, \infty]$ and all $i \in \mathcal{S}$,
\begin{equation}\label{uniformly-convex-concave-condition}
\begin{cases}
J_{T}(0,i; u_1(\cdot), 0) \geq \delta\, \mathbb{E} \int_0^T |u_1(t)|^2 dt, & \forall u_1(\cdot) \in \mathcal{U}_1[0,T], \\[3mm]
J_{T}(0,i; 0, u_2(\cdot)) \leq -\delta\, \mathbb{E} \int_0^T |u_2(t)|^2 dt, & \forall u_2(\cdot) \in \mathcal{U}_2[0,T].
\end{cases}
\end{equation}


\textbf{(A2).} The system $[A,C]_{\alpha}$ is $L^{2}$-stable.

\begin{remark}\rm
We refer to \eqref{uniformly-convex-concave-condition} as the uniform convexity-concavity assumption. As demonstrated in our prior research \cite{My-paper-2024-finite-M-ZLQ,My-paper-2024-infinite-M-ZLQ}, assumption (A1) is nearly necessary for solving Problem (M-ZLQ)$_{T}$. In contrast, assumption (A2) is introduced solely to ensure that the performance functional $J_{\infty}(x,i;u_{1}(\cdot),u_{2}(\cdot))$ is well-defined on $\mathbb{R}^{n}\times\mathcal{S}\times\mathcal{U}_{1}[0,\infty)\times \mathcal{U}_{2}[0,\infty)$.
\end{remark}

The following result offers refined estimates for the solution to SDE \eqref{AC-b-sigma}. 
\begin{proposition}\label{prop-AC-estimation}
Assuming (A1) holds, there exist constants $K, \mu > 0$ such that for any $b(\cdot), \sigma(\cdot) \in L_{\mathbb{F}}^{2}(\mathbb{R}^{n})$ and initial pair $(x,i) \in \mathbb{R}^{n} \times \mathcal{S}$, the solution $X(\cdot; x,i)$ to equation \eqref{AC-b-sigma} satisfies the estimates:
\begin{align}
\label{estimation-AC-bsigma-1}
&\mathbb{E}\left|X(t)\right|^{2}\leq K\left[e^{-\mu t}|x|^{2}+\mathbb{E}\int_{0}^{t}\left(|b(s)|^{2}+|\sigma(s)|^{2}\right)ds\right],\quad t\geq 0,\\
\label{estimation-AC-bsigma-2}
&\mathbb{E}\int_{0}^{t}\left|X(s)\right|^{2}ds\leq K\left[|x|^{2}+\mathbb{E}\int_{0}^{t}\left(|b(s)|^{2}+|\sigma(s)|^{2}\right)ds\right],\quad t\geq 0.
\end{align}
\end{proposition}
\begin{proof}
The first estimate \eqref{estimation-AC-bsigma-1} follows from the proof of Proposition 2.5 in \cite{My-paper-2025-infinite-M-SLQ} (see the last equation on page 860), while the second estimate \eqref{estimation-AC-bsigma-2} is obtained by integrating both sides of the first.
\end{proof}

As a consequence of Proposition \ref{prop-AC-estimation}, the following result gives an equivalent characterization of the $L^{2}$-stability, also known as the \emph{mean-square exponential stability} of $[A,C]_{\alpha}$.

\begin{corollary}\label{coro-AC}
System $[A,C]_{\alpha}$ is $L^{2}$-stable if and only if there exist constants $K,\, \mu>0$ such that
\begin{equation}\label{estimation-AC}
\mathbb{E}\left|X(t;x,i)\right|^{2}\leq Ke^{-\mu t}|x|^{2},\quad \forall t\geq 0,\quad \forall (x,i)\in\mathbb{R}^{n}\times\mathcal{S}.
\end{equation}
\end{corollary}
\begin{proof}
Clearly, \eqref{estimation-AC} implies \eqref{L2-stable}. Conversely, if the system $[A,C]_{\alpha}$ is $L^{2}$-stable, then by setting $b(\cdot)=\sigma(\cdot)=0$ in \eqref{estimation-AC-bsigma-1}, it follows that $[A,C]_{\alpha}$ is mean-square exponentially stable.
\end{proof}


We conclude this section by introducing notations that will be frequently used in the following sections. For a given $\mathbf{P}(\cdot) \in \mathcal{D}\left(C([0,T];\mathbb{S}^{n})\right)$, we define:
 \begin{equation}\label{notation-MLN-T}
 \left\{
\begin{array}{l}
\mathcal{M}(t;\mathbf{P}, i)\triangleq P(t,i)A(i)+A(i)^{\top}P(t,i)+C(i)^{\top}P(t,i)C(i)+Q(i)+\sum_{j=1}^{L}\pi_{ij}P(t,j)\\[3mm]
\mathcal{L}(t;\mathbf{P}, i)\triangleq\left(\mathcal{L}_{1}(t;\mathbf{P}, i),\mathcal{L}_{2}(t;\mathbf{P}, i)\right)= P(t,i)B(i)+C(i)^{\top}P(t,i)D(i)+S(i)^{\top}\\[3mm]
\mathcal{N}(t;\mathbf{P}, i)\triangleq \left(\begin{array}{cc}
\mathcal{N}_{11}(t;\mathbf{P}, i) &\mathcal{N}_{12}(t;\mathbf{P}, i)\\\mathcal{N}_{21}(t;\mathbf{P}, i) &\mathcal{N}_{22}(t;\mathbf{P}, i)
\end{array}\right) = D(i)^{\top}P(t,i)D(i)+R(i), \quad i\in\mathcal{S},
\end{array}
\right.
 \end{equation}
 with
 \begin{equation}\label{notation-MLN-T-2}
 \left\{
\begin{array}{l}
\mathcal{L}_{k}(t;\mathbf{P}, i)\triangleq P(t,i)B_{k}(i)+C(i)^{\top}P(t,i)D_{k}(i)+S_{k}(i)^{\top}\\[3mm]
\mathcal{N}_{kl}(t;\mathbf{P}, i)\triangleq D_{k}(i)^{\top}P(t,i)D_{l}(i)+R_{kl}(i),\quad k,l\in\{1,2\},\quad i\in \mathcal{S}.
\end{array}
\right.
 \end{equation}
 Furthermore, we also introduce similar notations for Problem (M-ZLQ)$_{\infty}$. For a given $\mathbf{P}\in\mathcal{D}(\mathbb{S}^{n})$, we define
 \begin{equation}\label{notation-MLN-infinite}
 \left\{
\begin{array}{l}
\mathcal{M}(\mathbf{P}, i)\triangleq P(i)A(i)+A(i)^{\top}P(i)+C(i)^{\top}P(i)C(i)+Q(i)+\sum_{j=1}^{L}\pi_{ij}P(j)\\[3mm]
\mathcal{L}(\mathbf{P}, i)\triangleq \left(\mathcal{L}_{1}(P(i),\mathcal{L}_{2}(\mathbf{P}, i)\right)= P(i)B(i)+C(i)^{\top}P(i)D(i)+S(i)^{\top}\\[3mm]
\mathcal{N}(\mathbf{P}, i)\triangleq \left(\begin{array}{cc}
\mathcal{N}_{11}(\mathbf{P}, i) &\mathcal{N}_{12}(\mathbf{P}, i)\\\mathcal{N}_{21}(\mathbf{P}, i) &\mathcal{N}_{22}(\mathbf{P}, i)
\end{array}\right)= D(i)^{\top}P(i)D(i)+R(i), \quad i\in\mathcal{S},
\end{array}
\right.
 \end{equation}
 with
 \begin{equation}\label{notation-MLN-infinite-2}
 \left\{
\begin{array}{l}
\mathcal{L}_{k}(\mathbf{P}, i)\triangleq P(i)B_{k}(i)+C(i)^{\top}P(i)D_{k}(i)+S_{k}(i)^{\top}\\[3mm]
\mathcal{N}_{kl}(\mathbf{P}, i)\triangleq D_{k}(i)^{\top}P(i)D_{l}(i)+R_{kl}(i),\quad k,l\in\{1,2\},\quad i\in \mathcal{S}.
\end{array}
\right.
 \end{equation}

\section{The unique solvability of zero-sum SLQ differential game}\label{section-3}
In this section, we establish the uniqueness and existence of an open-loop saddle strategy for Problem (M-ZLQ)$_{T}$ and Problem (M-ZLQ)$_{\infty}$ under assumptions (A1)-(A2). The following result essentially is taken from Wu et al. \cite{My-paper-2024-finite-M-ZLQ}.

\begin{theorem}\label{thm-T-open-loop-solvable}
Assume $(A1)$ holds. Then the following statements are true: 
\begin{description}
  \item[(i)] A control pair $\bar{u}_{T}(\cdot)=(\bar{u}_{1,T}(\cdot)^{\top},\bar{u}_{2,T}(\cdot)^{\top})^{\top}\in\mathcal{U}[0,T]$ is an open-loop saddle strategy of Problem (M-ZLQ)$_{T}$ for the initial pair $(x,i)\in\mathbb{R}^{n}\times\mathcal{S}$ if and only if the adapted solution $\left(\bar{X}_{T}(\cdot),\bar{Y}_{T}(\cdot),\bar{Z}_{T}(\cdot),\bar{\mathbf{\Gamma}}_{T}(\cdot)\right)$ to the forward backward stochastic differential equations (FBSDEs):
  \begin{equation}\label{FBSDEs-T}
  \left\{
      \begin{aligned}
      d\bar{X}_{T}(t)&=\left[A(\alpha_{t})\bar{X}_{T}(t)+B(\alpha_{t})\bar{u}_{T}(t)\right]dt+\left[C(\alpha_{t})\bar{X}_{T}(t)+D(\alpha_{t})\bar{u}_{T}(t)\right]dW(t),\\
      d\bar{Y}_{T}(t)&=-\left[A(\alpha_{t})^{\top}\bar{Y}_{T}(t)+C(\alpha_{t})^{\top}\bar{Z}_{T}(t)+Q(\alpha_{t})\bar{X}_{T}(t)+S(\alpha_{t})^{\top}\bar{u}_{T}(t)\right]dt\\
      &\quad+\bar{Z}_{T}(t)dW(t)+\bar{\mathbf{\Gamma}}_{T}(t)\cdot d\mathbf{\widetilde{N}}(t),\quad t\in [0,T],\\
     \bar{X}_{T} (0)&=x,\quad\alpha_{0}=i,
      \end{aligned}
      \right.
  \end{equation}
  satisfies the stationary condition:
  \begin{equation}\label{stationary-condition-T}
     B(\alpha_{t})^{\top}\bar{Y}_{T}(t)+ D(\alpha_{t})^{\top}\bar{Z}_{T}(t)+S(\alpha_{t})\bar{X}_{T}(t)+R(\alpha_{t})\bar{u}_{T}(t)=0,\quad a.e.\quad a.s..
  \end{equation}
    \item[(ii)] The following coupled differential Riccati equations
    \begin{equation}\label{CDREs-T}
   \left\{
    \begin{aligned}
    &\dot{P}_{T}(t,i)+\mathcal{M}(t;\mathbf{P_{T}},i)-\mathcal{L}(t;\mathbf{P_{T}},i)\mathcal{N}(t;\mathbf{P_{T}},i)^{-1}\mathcal{L}(t;\mathbf{P_{T}},i)^{\top}=0,\\
    &P(T,i)=0,\quad i\in\mathcal{S},
    \end{aligned}
    \right.
\end{equation}
admits a solution $\mathbf{P_{T}}(\cdot)=(P_T(\cdot,1),\cdots,P_T(\cdot,L))\in \mathcal{D}\left(C([0,T];\mathbb{S}^{n})\right)$ such that
\begin{equation}\label{strongly-regular-condition-T}
 \mathcal{N}_{11}(t;\mathbf{P_{T}},i)\geq \delta I ,\qquad \mathcal{N}_{22}(t;\mathbf{P_{T}},i)\leq -\delta I,\quad i\in\mathcal{S}.
\end{equation}
  \item[(iii)]  Let $\mathbf{P_{T}}(\cdot)\in \mathcal{D}\left(C([0,T];\mathbb{S}^{n})\right)$  be the solution of \eqref{CDREs-T} and set
  \begin{equation}\label{Theta-T}
    \Theta_{T}(t,i)\triangleq - \mathcal{N}(t;\mathbf{P_{T}},i)^{-1}\mathcal{L}(t;\mathbf{P_{T}},i)^{\top}, \quad t\in[0,T],\quad i\in\mathcal{S}.
  \end{equation}
  Then for any initial pair $(x,i)$,
  \begin{equation}\label{closed-loop-representation-T}
  \bar{u}_{T}(t)\triangleq \Theta_{T}(t,\alpha_{t})\bar{X}_{T}(t;x,i),\quad t\in[0,T],
 \end{equation}
 with $\bar{X}_{T}(\cdot;x,i)$ being the solution of the following closed-loop system:
  \begin{equation}\label{closed-loop-state-T}
   \left\{
   \begin{aligned}
   &d\bar{X}_{T}(t)=\left[A(\alpha_{t})+B(\alpha_{t})\Theta_{T}(t,\alpha_{t})\right]\bar{X}_{T}(t)dt
   +\left[C(\alpha_{t})+D(\alpha_{t})\Theta_{T}(t,\alpha_{t})\right]\bar{X}_{T}(t)dW(t)\\
   &\bar{X}_{T}(0)=x,\quad \alpha_{0}=i,
   \end{aligned}
   \right.
  \end{equation}
 is an open-loop saddle strategy of Problem (M-ZLQ)$_{T}$ for the initial pair $(x,i)$.
\end{description}
\end{theorem}

\begin{remark}\rm
 We call a solution $\mathbf{P_{T}}(\cdot) \in \mathcal{D}\bigl(C([0,T];\mathbb{S}^{n})\bigr)$ to the CDREs \eqref{CDREs-T} strongly regular if it satisfies condition \eqref{strongly-regular-condition-T}, which also implies that \(\mathcal{N}(t; \mathbf{P_{T}}, i)\) is invertible.
\end{remark}

The following result establishes the open-loop solvability of Problem (M-ZLQ)$_{\infty}$, which can be found in \citet{My-paper-2024-infinite-M-ZLQ}.

\begin{theorem}\label{thm-open-loop-solvable-infty}
  Assume (A1)-(A2) hold. Then the following statements are true:
  \begin{description}
    \item[(i)] A control pair $\bar{u}_{\infty}(\cdot)\triangleq (\bar{u}_{1,\infty}(\cdot)^{\top},\bar{u}_{2,\infty}(\cdot)^{\top})^{\top}\in\mathcal{U}_{ad}(x,i)$ is an open-loop saddle strategy of Problem (M-ZLQ)$_{\infty}$ for the initial pair $(x,i)$ if and only if the adapted solution $\left(\bar{X}_{\infty}(\cdot),\bar{Y}_{\infty}(\cdot),\bar{Z}_{\infty}(\cdot),\bar{\mathbf{\Gamma}}_{\infty}(\cdot)\right)$ to the FBSDEs:
  \begin{equation}\label{FBSDEs-infinite}
  \left\{
      \begin{aligned}
      d\bar{X}_{\infty}(t)&=\left[A(\alpha_{t})\bar{X}_{\infty}(t)+B(\alpha_{t})\bar{u}_{\infty}(t)\right]dt+\left[C(\alpha_{t})\bar{X}_{\infty}(t)+D(\alpha_{t})\bar{u}_{\infty}(t)\right]dW(t),\\
      d\bar{Y}_{\infty}(t)&=-\left[A(\alpha_{t})^{\top}\bar{Y}_{\infty}(t)+C(\alpha_{t})^{\top}\bar{Z}_{\infty}(t)+Q(\alpha_{t})\bar{X}_{\infty}(t)+S(\alpha_{t})^{\top}\bar{u}_{\infty}(t)\right]dt\\
      &\quad+\bar{Z}_{\infty}(t)dW(t)+\bar{\mathbf{\Gamma}}_{\infty}(t)\cdot d\mathbf{\widetilde{N}}(t),\quad t\geq 0,\\
     \bar{X}_{\infty} (0)&=x,\quad\alpha_{0}=i,
      \end{aligned}
      \right.
  \end{equation}
  satisfies the stationary condition:
  \begin{equation}\label{stationary-condition-infinite}
     B(\alpha_{t})^{\top}\bar{Y}_{\infty}(t)+ D(\alpha_{t})^{\top}\bar{Z}_{\infty}(t)+S(\alpha_{t})\bar{X}_{\infty}(t)+R(\alpha_{t})\bar{u}_{\infty}(t)=0,\quad a.e.\quad a.s..
  \end{equation}
    \item[(ii)] If the following coupled algebra Riccati equations (CAREs)
    \begin{equation}\label{CAREs-infinite}
    \mathcal{M}(\mathbf{P_{\infty}},i)-\mathcal{L}(\mathbf{P_{\infty}},i) \mathcal{N}(\mathbf{P_{\infty}},i)^{-1} \mathcal{L}(\mathbf{P_{\infty}},i)^{\top} = 0,
    \end{equation}
    admits a solution $\mathbf{P_{\infty}}\triangleq (P_{\infty}(1),\cdots, P_{\infty}(L))\in\mathcal{D}(\mathbb{S}^{n})$ such that
    \begin{equation}\label{Theta-infinity}
     \mathbf{\Theta}\in \mathcal{H}[A,C;B,D]_{\alpha},\quad \text{with}\quad \Theta(i)\triangleq -\mathcal{N}(\mathbf{P_{\infty}},i)^{-1} \mathcal{L}(\mathbf{P_{\infty}},i)^{\top},\quad i\in\mathcal{S},
    \end{equation}
    then for any initial pair $(x,i)$,
  \begin{equation}\label{closed-loop-representation-infty}
  \bar{u}_{\infty}(t)\triangleq \Theta_{\infty}(\alpha_{t})\bar{X}_{\infty}(t;x,i),\quad t>0,
 \end{equation}
    with $\bar{X}_{\infty}(\cdot;x,i)$ being the solution of the following closed-loop system:
  \begin{equation}\label{closed-loop-state-infinite}
   \left\{
   \begin{aligned}
   &d\bar{X}_{\infty}(t)=\left[A(\alpha_{t})+B(\alpha_{t})\Theta_{\infty}(\alpha_{t})\right]\bar{X}_{\infty}(t)dt
   +\left[C(\alpha_{t})+D(\alpha_{t})\Theta_{\infty}(\alpha_{t})\right]\bar{X}_{\infty}(t)dW(t)\\
   &\bar{X}_{\infty}(0)=x,\quad \alpha_{0}=i,
   \end{aligned}
   \right.
  \end{equation}
    is an open-loop saddle strategy of Problem (M-ZLQ)$_{\infty}$ for initial pair $(x,i)$.
  \end{description}
\end{theorem}

\begin{remark}\label{rmk-Y-infinte}
  Suppose that the control pair $\bar{u}_{\infty}(\cdot)$ in Theorem \ref{thm-open-loop-solvable-infty} is an open-loop saddle strategy for the initial pair $(x,i)$. Let $\varphi(\cdot)\triangleq Q(\alpha(\cdot))\bar{X}_{\infty}(\cdot)+S(\alpha(\cdot))^{\top}\bar{u}_{\infty}(\cdot)$ and by Proposition \ref{prop-AC-estimation}, we have $\varphi(\cdot)\in L_{\mathbb{F}}^{2}(\mathbb{R}^{n})$.
  Thus, equation (22) in the proof of \cite[Lemma 2.8]{My-paper-2025-infinite-M-SLQ} implies that the $L^{2}$-stable adapted solution of \eqref{FBSDEs-infinite} satisfies the following property:
  \begin{equation}\label{Y-limite-propety}
    \lim_{t\rightarrow\infty}\mathbb{E} \left|\bar{Y}_{\infty}(t)\right|^{2}=0.
  \end{equation}
\end{remark}

\begin{remark}\label{rmk-question}\rm
We note that although preliminary solutions to Problem (M-ZLQ)$_{T}$ and (M-ZLQ)$_{\infty}$ were established in our earlier work \cite{My-paper-2024-finite-M-ZLQ,My-paper-2024-infinite-M-ZLQ}, several substantive issues persist that merit deeper analysis:
\begin{description}
  \item[Q1.] Does a unique open-loop saddle strategy exist for Problems (M-ZLQ)$_{T}$ and (M-ZLQ)$_{\infty}$ for every initial pair $(x,i)\in\mathbb{R}^{n}\times \mathcal{S}$?
  \item[Q2.] Do the CDREs \eqref{CDREs-T} admit a unique strongly regular solution $\mathbf{P_{T}}(\cdot)\in\mathcal{D}\left(C([0,T];\mathbb{S}^{n})\right)$?
  \item[Q3.] Do the CAREs \eqref{CAREs-infinite} admit a solution $\mathbf{P_{\infty}}\in\mathcal{D}(\mathbb{S}^{n})$ satisfying condition \eqref{Theta-infinity}?
\end{description}
\end{remark}

To address question \textbf{Q1} in Remark \ref{rmk-question}, we re-examine the performance functional through Hilbert space operators. Consider the following SDE:
 \begin{equation}\label{state-k}
   \left\{
   \begin{aligned}
   dX_{k}(t)&=\left[A(\alpha_{t})X_{k}(t)+B_{k}(\alpha_{t})u_{k}(t)\right]dt+\left[C(\alpha_{t})X_{k}(t)+D_{k}(\alpha_{t})u_{k}(t)\right]dW(t),\qquad t\geq0,\\
   X_{k}(0)&=0,\quad \alpha_{0}=i,\quad k=1,2,
   \end{aligned}
   \right.
 \end{equation}
 and
  \begin{equation}\label{state-0}
   \left\{
   \begin{aligned}
   dX_{0}(t)&=A(\alpha_{t})X_{0}(t)dt+C(\alpha_{t})X_{0}(t)dW(t),\qquad t\geq0,\\
   X_{0}(0)&=x,\quad \alpha_{0}=i.
   \end{aligned}
   \right.
 \end{equation}
Since the system $[A,C]_{\alpha}$ is $L^{2}$-stable, Proposition \ref{prop-AC-estimation} guarantees a constant $K > 0$ (independent of $T$) such that the solution $X_{k}(\cdot)$ to \eqref{state-k} satisfies:
\begin{equation}\label{state-k-estimate}
\mathbb{E}\int_{0}^{T}\left|X_{k}(t)\right|^{2}dt\leq K\mathbb{E}\int_{0}^{T}\left|u_{k}(t)\right|^{2}dt,\quad \forall u_{k}(\cdot)\in \mathcal{U}_{k}[0,T],\quad T>0,\quad k=1,2,
\end{equation}
and the solution $X_{0}(\cdot)$ to \eqref{state-0} satisfies
\begin{equation}\label{state-0-estimate}
\mathbb{E}\int_{0}^{T}\left|X_{0}(t)\right|^{2}dt\leq K\left|x\right|^{2},\quad \forall  T>0.
\end{equation}

We define
\begin{equation}\label{operator}
\begin{array}{lll}
 \big[\mathcal{L}_{k,T}^{\alpha}u_{k}\big](\cdot)\triangleq X_{k}(\cdot),& \widehat{\mathcal{L}}_{k,T}^{\alpha}u_{k}\triangleq X_{k}(T),& k=1,2,\\[3mm]
 \big[\mathcal{N}_{T}^{\alpha}x\big](\cdot)\triangleq X_{0}(\cdot),& \widehat{\mathcal{N}}_{T}^{\alpha}x\triangleq X_{0}(T).
  \end{array}
\end{equation}
Clearly, for $k=1,2$,
$$
\mathcal{L}_{k,T}^{\alpha}: \mathcal{U}_{k}[0,T]\rightarrow L_{\mathbb{F}}^{2}([0,T],\mathbb{R}^{n}), \quad
\widehat{\mathcal{L}}_{k,T}^{\alpha}: \mathcal{U}_{k}[0,T]\rightarrow  L_{\mathcal{F}_{T}}^{2}(\Omega,\mathbb{R}^{n}),
$$
and
$$
\mathcal{N}_{T}^{\alpha}: \mathbb{R}^{n}\rightarrow L_{\mathbb{F}}^{2}([0,T],\mathbb{R}^{n}), \quad
\widehat{\mathcal{N}}_{T}^{\alpha}: \mathbb{R}^{n}\rightarrow L_{\mathcal{F}_{T}}^{2}(\Omega,\mathbb{R}^{n}),
$$
are linear operators and uniformly bounded in $T$. Based on these notations and the linearity of \eqref{state-k} and \eqref{state-0}, the state process defined in \eqref{ZLQ-state} can be decomposed into
$$
X(\cdot)=X_{0}(\cdot)+X_{1}(\cdot)+X_{2}(\cdot)=\big[\mathcal{N}_{T}^{\alpha}x\big](\cdot)+\big[\mathcal{L}_{1,T}^{\alpha}u_{1}\big](\cdot)+\big[\mathcal{L}_{2,T}^{\alpha}u_{2}\big](\cdot),
$$
and in particular,
$$
X(T)=X_{0}(T)+X_{1}(T)+X_{2}(T)=\widehat{\mathcal{N}}_{T}^{\alpha}x+\widehat{\mathcal{L}}_{1,T}^{\alpha}u_{1}+\widehat{\mathcal{L}}_{2,T}^{\alpha}u_{2}.
$$
Denote by $\mathcal{A}^{*}$ the adjoint operator of a linear operator $\mathcal{A}$. Then the performance functional \eqref{ZLQ-performance} can be represented as follows:
\begin{equation}\label{performance-functional-Hilbert}
J_{T}(x,i;u_{1}(\cdot),u_{2}(\cdot))=\big<\mathcal{M}_{T}^{\alpha}u,u\big>+2\big<\mathcal{K}_{T}^{\alpha}x,u\big>+\big<\mathcal{O}_{T}^{\alpha}x,x\big>,
\end{equation}
where $u(\cdot)=(u_{1}(\cdot)^{\top},u_{2}(\cdot)^{\top})^{\top}$ and
\begin{equation}\label{performance-functional-operator-1}
\mathcal{M}_{T}^{\alpha}\triangleq \left(
\begin{matrix}
\mathcal{M}_{11,T}^{\alpha} & \mathcal{M}_{12,T}^{\alpha}\\ \mathcal{M}_{21,T}^{\alpha} & \mathcal{M}_{22,T}^{\alpha}
\end{matrix}
\right),\quad
\mathcal{K}_{T}^{\alpha}\triangleq \left(
\begin{matrix}
\mathcal{K}_{1,T}^{\alpha} \\ \mathcal{K}_{2,T}^{\alpha}
\end{matrix}
\right),\quad
\mathcal{O}_{T}^{\alpha}\triangleq (\mathcal{N}_{T}^{\alpha})^{*}Q(\alpha)\mathcal{N}_{T}^{\alpha},
\end{equation}
with
\begin{equation}\label{performance-functional-operator-2}
\begin{aligned}
&\mathcal{M}_{ij,T}^{\alpha}\triangleq R_{ij}(\alpha)+S_{i}(\alpha)\mathcal{L}_{j,T}^{\alpha}+(\mathcal{L}_{j,T}^{\alpha})^{*}S_{i}(\alpha)^{\top}+(\mathcal{L}_{i,T}^{\alpha})^{*}Q(\alpha)\mathcal{L}_{j,T}^{\alpha},\quad i,j=1,2\\
&\mathcal{K}_{i,T}^{\alpha}\triangleq (\mathcal{L}_{i,T}^{\alpha})^{*}Q(\alpha)\mathcal{N}_{T}^{\alpha}+S_{i}(\alpha)\mathcal{N}_{T}^{\alpha}, \quad i=1,2.
\end{aligned}
\end{equation}
Note that the linear operators
$$
\mathcal{M}_{T}^{\alpha}: \,\,\mathcal{U}[0,T]\rightarrow  \mathcal{U}[0,T],\quad \mathcal{K}_{T}^{\alpha}:\,\, \mathbb{R}^{n}\rightarrow  \mathcal{U}[0,T],\quad \mathcal{O}_{T}^{\alpha}\in\mathbb{S}^{n},
$$
are all uniformly bounded in $T$ , and $\mathcal{M}_{T}^{\alpha}$ is self-adjoint.

Similar to the previous discussion, replacing the interval $[0, T ]$ by $[0, \infty)$, we can derive a similar operator representation for the performance functional $J_{\infty}(x,i;u_{1}(\cdot),u_{2}(\cdot))$:
\begin{equation}\label{performance-functional-Hilbert-infinite}
J_{\infty}(x,i;u_{1}(\cdot),u_{2}(\cdot))=\big<\mathcal{M}_{\infty}^{\alpha}u,u\big>+2\big<\mathcal{K}_{\infty}^{\alpha}x,u\big>+\big<\mathcal{O}_{\infty}^{\alpha}x,x\big>,
\end{equation}
where $\mathcal{O}_{\infty}^{\alpha}\in\mathbb{S}^{n}$, and the linear operators
$$
\mathcal{M}_{\infty}^{\alpha}\triangleq \left(
\begin{matrix}
\mathcal{M}_{11,\infty}^{\alpha} & \mathcal{M}_{12,\infty}^{\alpha}\\ \mathcal{M}_{21,\infty}^{\alpha} & \mathcal{M}_{22,\infty}^{\alpha}
\end{matrix}
\right):\,\, \mathcal{U}_{ad}(x,i)\rightarrow  \mathcal{U}_{ad}(x,i),\quad
\mathcal{K}_{\infty}^{\alpha}\triangleq \left(
\begin{matrix}
\mathcal{K}_{1,\infty}^{\alpha} \\ \mathcal{K}_{2,\infty}^{\alpha}
\end{matrix}
\right):\,\, \mathbb{R}^{n}\rightarrow \mathcal{U}_{ad}(x,i),
$$
are bounded with $\mathcal{M}_{\infty}^{\alpha}$ being self-adjoint.

Using the functional representations \eqref{performance-functional-Hilbert} and \eqref{performance-functional-Hilbert-infinite}, we derive the following existence and uniqueness result.

\begin{theorem}\label{thm-unique-solvability}
Let (A1)–(A2) hold. Then for each initial pair  $(x,i)$,  the following statements hold:
\begin{description}
  \item[(i)] Problem (M-ZLQ)$_{T}$ admits a unique open-loop saddle strategy for any $T>0$;
  \item[(ii)] Problem (M-ZLQ)$_{\infty}$  admits a unique open-loop saddle strategy.
\end{description}
\end{theorem}

\begin{proof}
Here, we only prove item (ii) since item (i) can be derived similarly. By definition, a pair \( \left(\bar{u}_{1,\infty}(\cdot), \bar{u}_{2,\infty}(\cdot)\right) \) is an open-loop saddle strategy if and only if
\begin{equation}\label{proof-unique-1}
\begin{aligned}
J_{\infty}\left(x,i ; \bar{u}_{1,\infty}(\cdot), \bar{u}_{2,\infty}(\cdot)+\varepsilon v_{2}(\cdot)\right) \leqslant  J_{\infty}\left(x,i ; \bar{u}_{1,\infty}(\cdot), \bar{u}_{2,\infty}(\cdot)\right) \leqslant J_{\infty}\left(x ,i; \bar{u}_{1,\infty}(\cdot)+\varepsilon v_{1}(\cdot), \bar{u}_{2,\infty}(\cdot)\right), \\
 \forall \varepsilon \in \mathbb{R}, \forall\left(v_{1}(\cdot), v_{2}(\cdot)\right) \in \mathscr{U}_{1}[0, \infty) \times \mathscr{U}_{2}[0, \infty) .
\end{aligned}
\end{equation}
From \eqref{performance-functional-Hilbert-infinite}, the performance functional $J_{\infty}\left(x,i ; u_{1}(\cdot), u_{2}(\cdot)\right)$ can be expressed as:
\begin{equation}\label{proof-unique-2}
\begin{aligned}
J_{\infty}\left(x,i ; u_{1}(\cdot), u_{2}(\cdot)\right)= & \left\langle\mathcal{M}_{11,\infty}^{\alpha} u_{1}, u_{1}\right\rangle+\left\langle\mathcal{M}_{22,\infty}^{\alpha} u_{2}, u_{2}\right\rangle+2\left\langle u_{1}, \mathcal{M}_{12,\infty}^{\alpha} u_{2}\right\rangle \\
& +2\left\langle u_{1}, \mathcal{K}_{1,\infty}^{\alpha} x\right\rangle+2\left\langle u_{2}, \mathcal{K}_{2,\infty}^{\alpha} x\right\rangle+\langle\mathcal{O}_{\infty}^{\alpha} x, x\rangle
\end{aligned}
\end{equation}
Therefore, 
\eqref{proof-unique-1} is equivalent to
\begin{equation}\label{proof-unique-3}
\left\{\begin{array}{ll}
\varepsilon^{2}\left\langle\mathcal{M}_{11,\infty}^{\alpha} v_{1}, v_{1}\right\rangle+2 \varepsilon\left\langle v_{1}, \mathcal{M}_{11,\infty}^{\alpha} \bar{u}_{1}+\mathcal{M}_{12,\infty}^{\alpha} \bar{u}_{2}+\mathcal{K}_{1,\infty}^{\alpha} x\right\rangle \geqslant 0, & \forall \varepsilon \in \mathbb{R}, \quad \forall v_{1}(\cdot) \in \mathscr{U}_{1}[0, \infty), \\
\varepsilon^{2}\left\langle\mathcal{M}_{22,\infty}^{\alpha} v_{2}, v_{2}\right\rangle+2 \varepsilon\left\langle v_{2}, \mathcal{M}_{11,\infty}^{\alpha} \bar{u}_{1}+\mathcal{M}_{12,\infty}^{\alpha} \bar{u}_{2}+\mathcal{K}_{1,\infty}^{\alpha} x\right\rangle \leqslant 0, & \forall \varepsilon \in \mathbb{R}, \quad \forall v_{2}(\cdot) \in \mathscr{U}_{2}[0, \infty) .
\end{array}\right.
\end{equation}

Observe that assumption (A1) is equivalent to the uniform positivity of \( \mathcal{M}_{11, \infty}^{\alpha}\) and \( -\mathcal{M}_{22, \infty}^{\alpha} \), that is, there exists a constant \( \delta>0 \) such that
\[
\left\langle\mathcal{M}_{11, \infty}^{\alpha} u_{1}, u_{1}\right\rangle \geqslant \delta\left\|u_{1}(\cdot)\right\|^{2}, \quad\left\langle\mathcal{M}_{22, \infty}^{\alpha} u_{2}, u_{2}\right\rangle \leqslant-\delta\left\|u_{2}(\cdot)\right\|^{2},
\]
for all \( u_{i}(\cdot) \in \mathscr{U}_{i}[0, \infty), i=1,2 \). Consequently, the equation \eqref{proof-unique-3} is in turn equivalent to
\[
\left\{\begin{array}{l}
\mathcal{M}_{11,\infty}^{\alpha} \bar{u}_{1,\infty}+\mathcal{M}_{12,\infty}^{\alpha} \bar{u}_{2,\infty}+\mathcal{K}_{1,\infty}^{\alpha} x=0, \\
\mathcal{M}_{11,\infty}^{\alpha} \bar{u}_{1,\infty}+\mathcal{M}_{12,\infty}^{\alpha} \bar{u}_{2,\infty}+\mathcal{K}_{1,\infty}^{\alpha} x=0.
\end{array}\right.
\]

Note that the operator \( \mathcal{M}_{\infty}^{\alpha} \) is invertible and the inverse \( (\mathcal{M}_{\infty}^{\alpha})^{-1} \) is given by
\[
(\mathcal{M}_{\infty}^{\alpha})^{-1} =\left(\begin{array}{cc}
(\mathcal{M}_{11,\infty}^{\alpha})^{-1}+\mathcal{H}_{\infty}^{\alpha} (\Phi_{\infty}^{\alpha})^{-1}(\mathcal{H}_{\infty}^{\alpha})^{*} & -\mathcal{H}_{\infty}^{\alpha} (\Phi_{\infty}^{\alpha})^{-1} \\
-(\Phi_{\infty}^{\alpha})^{-1}(\mathcal{H}_{\infty}^{\alpha})^{*} & (\Phi_{\infty}^{\alpha})^{-1}
\end{array}\right),
\]
where \( \Phi_{\infty}^{\alpha}\triangleq \mathcal{M}_{22,\infty}^{\alpha}-\mathcal{M}_{21,\infty}^{\alpha}(\mathcal{M}_{11,\infty}^{\alpha})^{-1} \mathcal{M}_{12,\infty}^{\alpha} \) is a negative (and thus invertible) operator, and $\mathcal{H}_{\infty}^{\alpha}\triangleq (\mathcal{M}_{11,\infty}^{\alpha})^{-1}\mathcal{M}_{12,\infty}^{\alpha}$. This implies that Problem (M-ZLQ)$_{\infty}$ admits a unique open-loop saddle strategy given by
\[
\bar{u}_{\infty}(\cdot)\triangleq \binom{\bar{u}_{1,\infty}(\cdot)}{\bar{u}_{2,\infty}(\cdot)}=-(\mathcal{M}_{\infty}^{\alpha})^{-1} \mathcal{K}_{\infty}^{\alpha}x.
\]
\end{proof}

The following result establishes the uniqueness of a strongly regular solution to CDREs \eqref{CDREs-T}.
\begin{proposition}\label{prop-CDREs-unique-T}
Let (A1) hold. Then the CDREs \eqref{CDREs-T} admit a unique strongly regular solution $\mathbf{P_{T}}(\cdot)\in \mathcal{D}\left(C([0,T];\mathbb{S}^{n})\right)$.
\end{proposition}
\begin{proof}
  Suppose both $\mathbf{P_{T}}(\cdot)$ and $\mathbf{\widehat{P}_{T}}(\cdot)$ are strongly regular solutions to CDREs \eqref{CDREs-T}. Consider Problem (M-ZLQ)$_{T}$ over the time horizon $[t,T]$ with initial value $(t,x,i)\in [0,T]\times \mathbb{R}^n \times \mathcal{S}$. Using a similar analysis as in \cite{My-paper-2024-finite-M-ZLQ}, we can obtain its value function given by:
  $$
  V_{T}(t,x,i)=\langle P_{T}(t,i)x,x\rangle=\langle\widehat{P}_{T}(t,i)x,x\rangle,\quad \forall (t,x,i)\in[0,T]\times \mathbb{R}^{n}\times \mathcal{S}.
  $$
  Hence, by the arbitrariness of $(t,x,i)$, we have
  $$
  P_{T}(t,i)=\widehat{P}_{T}(t,i),\quad \forall (t,i)\in[0,T]\times \mathcal{S}.
  $$
  This completes the proof.
\end{proof}

We have addressed question \textbf{Q2} from Remark \ref{rmk-question}. Moreover, combining Theorems \ref{thm-T-open-loop-solvable}, \ref{thm-unique-solvability}, and Proposition \ref{prop-CDREs-unique-T}, we obtain the following result.

\begin{corollary}\label{coro-T}
 Let assumption $(A1)$ hold. Then the following statements are true:
\begin{description}
  \item[(i)] The CDREs \eqref{CDREs-T} admit a unique regular solution $\mathbf{P}_{T} \in \mathcal{D}\bigl(C([0,T];\mathbb{S}^{n})\bigr)$;
  \item[(ii)] For any initial pair $(x,i)$, Problem (M-ZLQ)$_{T}$ admits a unique open-loop saddle strategy $\bar{u}_{T}(\cdot)$ with closed-loop representation given by \eqref{closed-loop-representation-T};
  \item[(iii)] The value function of Problem (M-ZLQ)$_{T}$ is
    $$
      V_{T}(x,i) = \langle P_{T}(0,i)x, x\rangle, \quad \forall (x,i)\in\mathbb R^{n}\times\mathcal S.
    $$
\end{description}
\end{corollary}

The analysis of question \textbf{Q3} from Remark \ref{rmk-question} is postponed to the next section, where we focus on the asymptotic properties of the open-loop saddle strategy for Problem (M-ZLQ)$_{T}$ as $T \to \infty$. As a byproduct, we establish the unique solvability of CAREs \eqref{CAREs-infinite}.

\section{Asymptotic properties of the open-loop saddle strategy}\label{section-4}
The following result establishes the convergence of Problem (M-ZLQ)$_{T}$ to Problem (M-ZLQ)$_{\infty}$ in an appropriate sense.
\begin{theorem}\label{thm-convergence}
Assume (A1)-(A2) hold. 
Let \( \left(\bar{u}_{1, T}(\cdot), \bar{u}_{2, T}(\cdot)\right) \) and \( \left(\bar{u}_{1,\infty}(\cdot), \bar{u}_{2,\infty}(\cdot)\right) \) denote the open-loop saddle strategies of Problem (M-ZLQ)$_{T}$ and Problem (M-ZLQ)$_{\infty}$, respectively. Then
\begin{equation}\label{strategy-convergence}
\lim _{T \rightarrow \infty} \mathbb{E} \int_{0}^{T}\left[\left|\bar{u}_{1, T}(t)-\bar{u}_{1,\infty}(t)\right|^{2}+\left|\bar{u}_{2, T}(t)-\bar{u}_{2,\infty}(t)\right|^{2}\right] d t=0.
\end{equation}
Consequently, for the corresponding state processes \( \bar{X}_{T}(\cdot) \) and \( \bar{X}_{\infty}(\cdot) \), we have
\begin{equation}\label{state-convergence}
\lim _{T \rightarrow \infty} \mathbb{E} \int_{0}^{T}\left|\bar{X}_{T}(t)-\bar{X}_{\infty}(t)\right|^{2} d t=0.
\end{equation}
\end{theorem}
\begin{proof}
 Let
 $$
 \bar{u}_{T}(\cdot)=\left(\begin{array}{c}\bar{u}_{1,T}(\cdot)\\ \bar{u}_{2,T}(\cdot)\end{array}\right),\quad  \bar{u}_{\infty}(\cdot)=\left(\begin{array}{c}\bar{u}_{1,\infty}(\cdot)\\ \bar{u}_{2,\infty}(\cdot)\end{array}\right).
 $$
 Then by Theorem \ref{thm-T-open-loop-solvable}, the adapted solution $(\bar{X}_{T}(\cdot),\bar{Y}_{T}(\cdot),\bar{Z}_{T}(\cdot),\mathbf{\bar{\Gamma}_{T}}(\cdot))$ to FBSDEs \eqref{FBSDEs-T} satisfies condition \eqref{stationary-condition-T}. Similarly, by Theorem \ref{thm-open-loop-solvable-infty}, the adapted solution $(\bar{X}_{\infty}(\cdot), \bar{Y}_{\infty}(\cdot), \bar{Z}_{\infty}(\cdot),\mathbf{\bar{\Gamma}_{\infty}}(\cdot))$ to FBSDEs \eqref{FBSDEs-infinite} satisfies condition \eqref{stationary-condition-infinite}.

Now, for $t\in[0,T]$, we define
 $$
 \begin{array}{l}
 \hat{u}_{T}(t)=\left(\begin{array}{c}\hat{u}_{1,T}(t)\\ \hat{u}_{2,T}(t)\end{array}\right)\triangleq \left(\begin{array}{c}\bar{u}_{1,\infty}(t)-\bar{u}_{1,T}(t)\\\bar{u}_{2,\infty}(t)- \bar{u}_{2,T}(t)\end{array}\right)
 =\bar{u}_{\infty}(t)-\bar{u}_{T}(t),\\[5mm]
 \hat{X}_{T}(t)=\bar{X}_{\infty}(t)-\bar{X}_{T}(t),\quad \hat{Y}_{T}(t)=\bar{Y}_{\infty}(t)-\bar{Y}_{T}(t),\quad \hat{Z}_{T}(t)=\bar{Z}_{\infty}(t)-\bar{Z}_{T}(t),\quad \mathbf{\hat{\Gamma}_{T}}(t)=\mathbf{\bar{\Gamma}_{\infty}}(t)-\mathbf{\bar{\Gamma}_{T}}(t).
 \end{array}
 $$
 Then, $(\hat{X}_{T}(\cdot),\hat{Y}_{T}(\cdot),\hat{Z}_{T}(\cdot),\mathbf{\hat{\Gamma}_{T}}(\cdot))$ solves the  FBSDEs:
 \begin{equation}\label{FBSDEs-T-2}
  \left\{
      \begin{aligned}
      d\hat{X}_{T}(t)&=\left[A(\alpha_{t})\hat{X}_{T}(t)+B(\alpha_{t})\hat{u}_{T}(t)\right]dt+\left[C(\alpha_{t})\hat{X}_{T}(t)+D(\alpha_{t})\hat{u}_{T}(t)\right]dW(t),\\
      d\hat{Y}_{T}(t)&=-\left[A(\alpha_{t})^{\top}\hat{Y}_{T}(t)+C(\alpha_{t})^{\top}\hat{Z}_{T}(t)+Q(\alpha_{t})\hat{X}_{T}(t)+S(\alpha_{t})^{\top}\hat{u}_{T}(t)\right]dt\\
      &\quad+\hat{Z}_{T}(t)dW(t)+\mathbf{\hat{\Gamma}_{T}}(t)\cdot d\mathbf{\widetilde{N}}(t),\quad t\in [0,T],\\
     \hat{X}_{T} (0)&=0,\quad\alpha_{0}=i,\quad \hat{Y}_{T}(T)=\bar{Y}_{\infty}(T),
      \end{aligned}
      \right.
  \end{equation}
  and satisfies the stationary condition:
  \begin{equation}\label{stationary-condition-T-2}
     B(\alpha_{t})^{\top}\hat{Y}_{T}(t)+ D(\alpha_{t})^{\top}\hat{Z}_{T}(t)+S(\alpha_{t})\hat{X}_{T}(t)+R(\alpha_{t})\hat{u}_{T}(t)=0,\quad a.e.\quad a.s..
  \end{equation}
Thus, by Theorem \ref{thm-T-open-loop-solvable} (i) and Theorem \ref{thm-unique-solvability}, we obtain that $\hat{u}_{T}(\cdot)$ is the unique open-loop saddle strategy for the zero-sum SLQ differential game with the state equation
\begin{equation}\label{ZLQ-state-2}
   \left\{
   \begin{aligned}
   dX(t)&=\left[A(\alpha_{t})X(t)+B_{1}(\alpha_{t})u_{1}(t)+B_{2}(\alpha_{t})u_{2}(t)\right]dt\\
   &\quad+\left[C(\alpha_{t})X(t)+D_{1}(\alpha_{t})u_{1}(t)+D_{2}(\alpha_{t})u_{2}(t)\right]dW(t),\qquad t\in [0,T],\\
   X(0)&=0,\quad \alpha(0)=i,
   \end{aligned}
   \right.
 \end{equation}
 and the performance functional
 \begin{equation}\label{ZLQ-performance-2}
\begin{aligned}
    \hat{J}_{T}\left(0,i;u_{1}(\cdot),u_{2}(\cdot)\right)
   & \triangleq \mathbb{E}\Bigg[\int_{0}^{T}
    \left<
    \left(
    \begin{array}{ccc}
    Q(\alpha_{t}) & S_{1}(\alpha_{t})^{\top} & S_{2}(\alpha_{t})^{\top} \\
    S_{1}(\alpha_{t}) & R_{11}(\alpha_{t}) & R_{12}(\alpha_{t}) \\
    S_{2}(\alpha_{t}) & R_{21}(\alpha_{t}) & R_{22}(\alpha_{t})
    \end{array}
    \right)
    \left(
    \begin{array}{c}
    X(t) \\
    u_{1}(t) \\
    u_{2}(t)
    \end{array}
    \right),
    \left(
     \begin{array}{c}
    X(t) \\
    u_{1}(t) \\
    u_{2}(t)
    \end{array}
    \right)
    \right>dt\\
    &\quad +2\left<\bar{Y}_{\infty}(T),X(t)\right>\Bigg].
  \end{aligned}
\end{equation}

Next, we represent the open-loop saddle strategy $\hat{u}_{T}(\cdot)$ of the above zero-sum SLQ differential game problem using Hilbert space operators. Using similar analysis as in Section \ref{section-3}, the performance functional \eqref{ZLQ-performance-2} can be represented as
 \begin{equation}\label{hat-J-representation}
  \hat{J}_{T}\left(0,i;u_{1}(\cdot),u_{2}(\cdot)\right)=\big<\mathcal{M}_{T}^{\alpha}u,u\big>+2\big<(\widehat{\mathcal{L}}_{T}^{\alpha})^{*}\bar{Y}_{\infty}(T),u\big>,
 \end{equation}
where $\widehat{\mathcal{L}}_{T}^{\alpha}\triangleq \left(\widehat{\mathcal{L}}_{1,T}^{\alpha},\widehat{\mathcal{L}}_{2,T}^{\alpha}\right)
$ and 
 $\mathcal{M}_{T}^{\alpha}=\left(
\begin{matrix}
\mathcal{M}_{11,T}^{\alpha} & \mathcal{M}_{12,T}^{\alpha}\\ \mathcal{M}_{21,T}^{\alpha} & \mathcal{M}_{22,T}^{\alpha}
\end{matrix}
\right)$
is defined in \eqref{performance-functional-operator-1}. Clearly, $\mathcal{M}_{T}^{\alpha}$ is invertible and its inverse $(\mathcal{M}_{T}^{\alpha})^{-1}$ is given by
 \[
(\mathcal{M}_{T}^{\alpha})^{-1} =\left(\begin{array}{cc}
(\mathcal{M}_{11,T}^{\alpha})^{-1}+\mathcal{H}_{T}^{\alpha} (\Phi_{T}^{\alpha})^{-1}(\mathcal{H}_{T}^{\alpha})^{*} & -\mathcal{H}_{T}^{\alpha} (\Phi_{T}^{\alpha})^{-1} \\
-(\Phi_{T}^{\alpha})^{-1}(\mathcal{H}_{T}^{\alpha})^{*} & (\Phi_{T}^{\alpha})^{-1}
\end{array}\right),
\]
where \( \Phi_{T}^{\alpha}\triangleq \mathcal{M}_{22,T}^{\alpha}-\mathcal{M}_{21,T}^{\alpha}(\mathcal{M}_{11,T}^{\alpha})^{-1} \mathcal{M}_{12,T}^{\alpha} \) is a negative (and thus invertible) operator, and $\mathcal{H}_{T}^{\alpha}\triangleq (\mathcal{M}_{11,T}^{\alpha})^{-1}\mathcal{M}_{12,T}^{\alpha}$. Consequently, equation \eqref{hat-J-representation} shows that the unique open-loop saddle strategy $\hat{u}_{T}(\cdot)$ can be represented as:
$$
\hat{u}_{T}(\cdot)=-(\mathcal{M}_{T}^{\alpha})^{-1}(\widehat{\mathcal{L}}_{T}^{\alpha})^{*}\bar{Y}_{\infty}(T).
$$

Observe that condition (A1) implies that for any $T>0$, there exists a constant $\delta>0$ (independent of $T$) such that
\[
\left\langle\mathcal{M}_{11, T}^{\alpha} u_{1}, u_{1}\right\rangle \geqslant \delta\left\|u_{1}(\cdot)\right\|^{2}, \quad\left\langle\mathcal{M}_{22, T}^{\alpha} u_{2}, u_{2}\right\rangle \leqslant-\delta\left\|u_{2}(\cdot)\right\|^{2},
\]
which further implies that $(\mathcal{M}_{11,T}^{\alpha})^{-1}$, $(\Phi_{T}^{\alpha})^{-1}$, and hence $(\mathcal{M}_{T}^{\alpha})^{-1}$ are uniformly bounded in $T$, that is, $
||(\mathcal{M}_{T}^{\alpha})^{-1}||\leq K$ for some constant $K>0$ (independent of $T$). Recall that the operator $\widehat{\mathcal{L}}_{T}^{\alpha}$ is also uniformly bounded in $T$. Thus, from the expression of $\hat{u}_{T}(\cdot)$, there exists a constant $K > 0$, independent of $T$, such that:
$$
\mathbb{E}\int_{0}^{T}\left[\left|\bar{u}_{1, T}(t)-\bar{u}_{1,\infty}(t)\right|^{2}+\left|\bar{u}_{2, T}(t)-\bar{u}_{2,\infty}(t)\right|^{2}\right] d t=||\hat{u}_{T}(\cdot)||\leq K\mathbb{E}\left|Y_{\infty}(T)\right|^{2}.
$$
Consequently, the desired result \eqref{strategy-convergence} follows from Remark \ref{rmk-Y-infinte}. Moreover, since $\hat{X}_{T}(\cdot)$ satisfies the first SDE in \eqref{FBSDEs-T-2} with initial value $\hat{X}_{T}(0)=0$, applying Proposition \ref{prop-AC-estimation} together with \eqref{strategy-convergence} yields \eqref{state-convergence}.
\end{proof}

Based on the above result, we further investigate the convergence of the solution to the CDREs \eqref{CDREs-T}. The following result establishes the convergence of $\lim_{T\rightarrow \infty}P_{T}(t;i)$ and provides some properties for the limit matrix $P(i)$, $i\in\mathcal{S}$.

\begin{proposition}\label{prop-CAREs-existence}
Let (A1)-(A2) hold and $\mathbf{P_{T}(\cdot)}$ be the unique strongly regular solution to CDREs \eqref{CDREs-T}. Then, for each $i\in\mathcal{S}$,  the limit
\begin{equation}\label{P-limit-2}
P(i)\triangleq \lim_{T\rightarrow \infty}P_{T}(t,i),\quad i\in\mathcal{S},
\end{equation}
exists independent of $t$ and $\mathbf{P}\triangleq \left(P(1),P(2),\cdots,P(L)\right)\in\mathcal{D}\left(\mathbb{S}^{n}\right)$ has the following properties:
\begin{description}
  \item[(i)] $\mathcal{N}_{11}(\mathbf{P},i)>0$, $\mathcal{N}_{22}(\mathbf{P},i)<0$;
  \item[(ii)] $\mathbf{P}$ is the unique solution to CAREs
  \begin{equation}\label{CAREs-infinite-2}
    \mathcal{M}(\mathbf{P},i)-\mathcal{L}(\mathbf{P},i) \mathcal{N}(\mathbf{P},i)^{-1} \mathcal{L}(\mathbf{P},i)^{\top} = 0,\quad i\in\mathcal{S}.
    \end{equation}
  \item [(iii)] Set $\Theta(i)\triangleq -\mathcal{N}(\mathbf{P},i)^{-1}\mathcal{L}(\mathbf{P},i)^{\top},\  i\in\mathcal{S}$ and let $\mathbf{\Theta}\triangleq \left(\Theta(1),\Theta(2),\cdots,\Theta(L)\right)$. Then, 
\begin{equation}\label{CAREs-constraint}
 \mathbf{\Theta}\in \mathcal{H}[A,C;B,D]_{\alpha}.
 \end{equation}
\end{description}
\end{proposition}
\begin{proof}
Clearly, for any $0\leq t\leq T$ and $0\leq s\leq T-t$, one has $P_{T}(t+s,i)=P_{T-t}(s,i), i\in\mathcal{S}$, which implies that
\begin{equation}\label{relation}
P_{T}(t,i)=P_{T-t}(0,i),\quad \forall 0\leq t\leq T<\infty, \quad i\in\mathcal{S}.
\end{equation}
Consequently, the limit if exists
$$\lim_{T\rightarrow \infty}P_{T}(t,i)=\lim_{T\rightarrow \infty}P_{T-t}(0,i)=\lim_{T\rightarrow \infty}P_{T}(0,i),\quad \forall i\in\mathcal{S},$$
is independent of $t$.

Now, we first prove that the limit $\lim_{T\rightarrow \infty}P_{T}(0,i)$ exist for every $i\in\mathcal{S}$. According to \cite[Theorem 4.1]{My-paper-2024-finite-M-ZLQ}, 
the value functions $V_{T}(\cdot,\cdot)$ of Problem (M-ZLQ)$_{T}$ is given by
$$
V_{T}(x,i)=\big<P_{T}(0,i)x,x\big>,\quad \forall (x,i)\in\mathbb{R}^{n}\times\mathcal{S}.
$$
Let  $\bar{u}_{T}(\cdot)=\left(\begin{array}{c}\bar{u}_{1,T}(\cdot)\\ \bar{u}_{2,T}(\cdot)\end{array}\right), \  \bar{u}_{\infty}(\cdot)=\left(\begin{array}{c}\bar{u}_{1,\infty}(\cdot)\\ \bar{u}_{2,\infty}(\cdot)\end{array}\right)
 $
 be the open-loop saddle strategies of Problem (M-ZLQ)$_{T}$ and Problem (M-ZLQ)$_{\infty}$, respectively. Then, by Theorem \ref{thm-convergence}, we have
 $$
 \begin{aligned}
\lim _{T \rightarrow \infty}\left\langle P_{T}(0,i) x, x\right\rangle&=\lim _{T \rightarrow \infty} V_{T}(x,i)=\lim _{T \rightarrow \infty} J_{T}\left(x ,i; \bar{u}_{1, T}(\cdot), \bar{u}_{2, T}(\cdot)\right) \\
&=\lim _{T \rightarrow \infty} \mathbb{E} \int_{0}^{T}\left\langle\left(\begin{array}{cc}
Q(\alpha_{t}) & S(\alpha_{t})^{\top} \\
S(\alpha_{t}) & R(\alpha_{t})
\end{array}\right)
\left(\begin{array}{c}\bar{X}_{T}(t)\\ \bar{u}_{T}(t)\end{array}\right),
\left(\begin{array}{c}\bar{X}_{T}(t)\\ \bar{u}_{T}(t)\end{array}\right)\right\rangle d t \\
&=\mathbb{E} \int_{0}^{\infty}\left\langle\left(\begin{array}{cc}
Q(\alpha_{t}) & S(\alpha_{t})^{\top} \\
S(\alpha_{t}) & R(\alpha_{t})
\end{array}\right)
\left(\begin{array}{c}\bar{X}_{\infty}(t)\\ \bar{u}_{\infty}(t)\end{array}\right),
\left(\begin{array}{c}\bar{X}_{\infty}(t)\\ \bar{u}_{\infty}(t)\end{array}\right)\right\rangle d t\\
& =J_{\infty}\left(x,i ; \bar{u}_{1,\infty}(\cdot), \bar{u}_{2,\infty}(\cdot)\right)=V_{\infty}(x,i),\quad \forall (x,i)\in\mathbb{R}^{n}\times\mathcal{S},
\end{aligned}
 $$
 which implies that the limit $\lim_{T\rightarrow \infty}P_{T}(0,i)$ exists for every $i\in\mathcal{S}$.

 Next, we show that the limit $\mathbf{P} \triangleq (P(1), P(2), \ldots, P(L)) \in \mathcal{D}(\mathbb{S}^{n})$ satisfies the properties stated in Proposition \ref{prop-CAREs-existence}. Property (i) follows directly from \eqref{strongly-regular-condition-T}. To prove property (ii)-(iii), we divide the argument into three steps.

 For step 1, we show that the limit $\mathbf{P}$ is the solution to \eqref{CAREs-infinite-2}. It follows from CDREs \eqref{CDREs-T} and relation \eqref{relation} that
 \begin{align*}
 P_{T}(1,i)-P_{T}(0,i)&=-\int_{0}^{1}\left[\mathcal{M}(t;\mathbf{P_{T}},i)-\mathcal{L}(t;\mathbf{P_{T}},i)\mathcal{N}(t;\mathbf{P_{T}},i)^{-1}\mathcal{L}(t;\mathbf{P_{T}},i)^{\top}\right]dt\\
 &=-\int_{0}^{1}\left[\mathcal{M}(0;\mathbf{P_{T-t}},i)-\mathcal{L}(0;\mathbf{P_{T-t}},i)\mathcal{N}(0;\mathbf{P_{T-t}},i)^{-1}\mathcal{L}(0;\mathbf{P_{T-t}},i)^{\top}\right]dt\\
 &=\int_{T-1}^{T}\left[\mathcal{M}(0;\mathbf{P_{t}},i)-\mathcal{L}(0;\mathbf{P_{t}},i)\mathcal{N}(0;\mathbf{P_{t}},i)^{-1}\mathcal{L}(0;\mathbf{P_{t}},i)^{\top}\right]dt.
 \end{align*}
 Letting $T\rightarrow\infty$, the above equation yields
 $$
 0=\mathcal{M}(\mathbf{P},i)-\mathcal{L}(\mathbf{P},i)\mathcal{N}(\mathbf{P},i)^{-1}\mathcal{L}(\mathbf{P},i)^{\top},\quad i\in\mathcal{S}.
 $$

 For step 2, we verify that the limit $\mathbf{P}$ satisfies condition \eqref{CAREs-constraint}. Applying Theorem \ref{thm-T-open-loop-solvable} to Problem (M-ZLQ)$_T$, the open-loop saddle strategy \( \left(\bar{u}_{1, T}(\cdot), \bar{u}_{2, T}(\cdot)\right) \) for initial pair $(x,i)$ admits the following closed-loop representation:
\[
\bar{u}_{T}(t) =\Theta_{T}(t,\alpha_{t}) \bar{X}_{T}(t;x,i), \quad t \in[0, T],
\]
where $\Theta_{T}(\cdot,i)$ and $\bar{X}_{T}(\cdot)\triangleq\bar{X}_{T}(\cdot;x,i)$ are defined in \eqref{Theta-T} and \eqref{closed-loop-state-T}, respectively. Note that the existence of the limit \(\lim_{T\to\infty} P_T(t,i) = \lim_{T\to\infty} P_{T-t}(0,i)\) implies that the following holds for some constant \(K > 0\):
\begin{equation}\label{PT-bound}
\left|P_{T}(t,i)\right|=\left|P_{T-t}(0,i)\right| \leq K, \quad \forall t\in[0,T],\quad i\in\mathcal{S}.
\end{equation}
Therefore, from the definition of $\Theta_{T}(t,i)$, we have 
\begin{equation}\label{Theta-T-bounded}
\left|\Theta_{T}(t,i)\right| \leq K, \quad \forall  t\in[0,T], \quad i\in\mathcal{S},
\end{equation}
for some possibly different constant \( K>0 \). Moreover, we also have 
\begin{equation}\label{Theta-T-limit}
\lim _{T \rightarrow \infty} \Theta_{T}(t,i)=-\mathcal{N}(\mathbf{P},i)^{-1}\mathcal{L}(\mathbf{P},i)^{\top}=\Theta(i),\quad i\in\mathcal{S}.
\end{equation}
Let \( \bar{X}_{\infty}(\cdot)\triangleq \bar{X}_{\infty}(\cdot;x,i)\in L_{\mathbb{F}}^{2}(\mathbb{R}^{n}) \) denote the state process corresponding to the saddle strategy \( \bar{u}_{\infty}(\cdot)\triangleq \bar{u}_{\infty}(\cdot;x,i) \) of Problem (M-ZLQ)$_{\infty}$ for initial pair $(x,i)$. Then by Theorem \ref{thm-convergence}, equations \eqref{Theta-T-bounded}-\eqref{Theta-T-limit},
\begin{align*}
&\quad \mathbb{E} \int_{0}^{\infty}|\bar{u}_{\infty}(t)-\Theta(\alpha_{t}) \bar{X}_{\infty}(t)|^{2} d t\\
&=\lim _{T \rightarrow \infty} \mathbb{E} \int_{0}^{T}|\bar{u}_{\infty}(t)-\Theta(\alpha_{t}) \bar{X}_{\infty}(t)|^{2} d t \\
&\leq 2 \lim _{T \rightarrow \infty}\left[\mathbb{E} \int_{0}^{T}\left|\bar{u}_{\infty}(t)-\bar{u}_{T}(t)\right|^{2} d t+\mathbb{E} \int_{0}^{T}\left|\bar{u}_{T}(t)-\Theta(\alpha_{t}) \bar{X}_{\infty}(t)\right|^{2} d t\right] \\
&=2 \lim _{T \rightarrow \infty} \mathbb{E} \int_{0}^{T}\left|\Theta_{T}(t,\alpha_{t}) \bar{X}_{T}(t)-\Theta(\alpha_{t}) \bar{X}_{\infty}(t)\right|^{2} d t \\
& \leq 4 \lim _{T \rightarrow \infty}\left[\mathbb{E} \int_{0}^{T}\left|\Theta_{T}(t,\alpha_{t})\right|^{2}\left|\bar{X}_{T}(t)-\bar{X}_{\infty}(t)\right|^{2} d t+\mathbb{E} \int_{0}^{T}\left|\Theta_{T}(t,\alpha_{t})-\Theta(\alpha_{t})\right|^{2}|\bar{X}_{\infty}(t)|^{2} d t\right]=0.
\end{align*}
This implies that $\bar{u}_{\infty}(t;x,i) =\Theta(\alpha_{t})\bar{X}_{\infty}(t;x,i)$ for every initial pair $(x,i)$. Consequently,  the process $\bar{X}_{\infty}(\cdot)\triangleq \bar{X}_{\infty}(\cdot;x,i)\in L_{\mathbb{F}}^{2}(\mathbb{R}^{n})$ solves 
the following SDE:
$$
\left\{
   \begin{aligned}
   &d\bar{X}_{\infty}(t)=\left[A(\alpha_{t})+B(\alpha_{t})\Theta(\alpha_{t})\right]\bar{X}_{\infty}(t)dt
   +\left[C(\alpha_{t})+D(\alpha_{t})\Theta(\alpha_{t})\right]\bar{X}_{\infty}(t)dW(t)\\
   &\bar{X}_{\infty}(0)=x,\quad \alpha_{0}=i.
   \end{aligned}
   \right.
$$
Therefore, by definition, $\mathbf{\Theta}\triangleq \left[\Theta(1),\Theta(2),\cdots,\Theta(L)\right]\in \mathcal{H}[A,C;B,D]_{\alpha}$ since $\bar{X}(\cdot;x,i)\in L_{\mathbb{F}}^{2}(\mathbb{R}^{n})$ for any initial pair $(x,i)$.

For step $3$, we prove the unique solvability of the constrained CAREs \eqref{CAREs-infinite-2}-\eqref{CAREs-constraint}. Let $\mathbf{P}$ be a solution to CAREs \eqref{CAREs-infinite-2} satisfying condition \eqref{CAREs-constraint}. If we show that the value function of Problem (M-ZLQ)$_{\infty}$ can be represented as
\begin{equation}\label{value-infinite}
V_{\infty}(x,i)=\left<P(i)x,x\right>,\quad \forall (x,i)\in\mathbb{R}^{n}\times \mathcal{S},
\end{equation}
then uniqueness follows immediately. To this end, let $\mathbf{\Theta}\in \mathcal{H}[A,C;B,D]_{\alpha}$ be defined by \eqref{CAREs-constraint} and let $X_{\Theta}(\cdot;x,i)$ denote the solution to
$$
\left\{
   \begin{aligned}
   &dX_{\Theta}(t)=\left[A(\alpha_{t})+B(\alpha_{t})\Theta(\alpha_{t})\right]X_{\Theta}(t)dt
   +\left[C(\alpha_{t})+D(\alpha_{t})\Theta(\alpha_{t})\right]X_{\Theta}(t)dW(t)\\
   &X_{\Theta}(0)=x,\quad \alpha_{0}=i.
   \end{aligned}
   \right.
$$
By \cite[Corollary 4.6]{My-paper-2024-infinite-M-ZLQ}, the control
$$
\bar{u}(\cdot;x,i)=\left(\begin{matrix}\bar{u}_{1}(\cdot)\\\bar{u}_{2}(\cdot)\end{matrix}\right)\triangleq \Theta(\alpha(\cdot))X_{\Theta}(\cdot;x,i),\quad (x,i)\in\mathbb{R}^{n}\times\mathcal{S},
$$
is the unique open-loop saddle strategy of Problem (M-ZLQ)$_{\infty}$. Consequently, we have
\begin{equation}\label{step-3-1}
\begin{aligned}
V_{\infty}(x,i)&=J(x,i;\bar{u}_{1}(\cdot),\bar{u}_{2}(\cdot))\\
&=\mathbb{E} \int_{0}^{\infty}\left\langle\left(\begin{matrix}
Q(\alpha_{t}) & S(\alpha_{t})^{\top} \\
S(\alpha_{t}) & R(\alpha_{t})
\end{matrix}\right)
\left(\begin{matrix}X_{\Theta}(t)\\ \Theta(\alpha_{t})X_{\Theta}(t)\end{matrix}\right),
\left(\begin{matrix}X_{\Theta}(t)\\ \Theta(\alpha_{t})X_{\Theta}(t)\end{matrix}\right)\right\rangle d t\\
&=\mathbb{E} \int_{0}^{\infty}\left\langle\left(Q(\alpha_{t})+S(\alpha_{t})^{\top}\Theta(\alpha_{t})+\Theta(\alpha_{t})^{\top}S(\alpha_{t})+\Theta(\alpha_{t})^{\top}R(\alpha_{t})\Theta(\alpha_{t})
\right)X_{\Theta}(t),X_{\Theta}(t)\right\rangle d t.
\end{aligned}
\end{equation}

Applying It\^o's rule to $\left<P(\alpha_{t})X_{\Theta}(t),X_{\Theta}(t)\right>$, we obtain
\begin{align*}
&\quad \mathbb{E}\left[\left<P(\alpha_{T})X_{\Theta}(T),X_{\Theta}(T)\right>-\left<P(i)x,x\right>\right]\\
&=\mathbb{E}\int_{0}^{T}\Big\{\Big<\big[P(\alpha_{t})A(\alpha_{t})+A(\alpha_{t})^{\top}P(\alpha_{t})+C(\alpha_{t})^{\top}P(\alpha_{t})C(\alpha_{t})+\sum_{j=1}^{L}\pi_{\alpha_{t}j}P(j)\big]X_{\Theta}(t),X_{\Theta}(t)\Big>\\
&\qquad+\Big<\big[P(\alpha_{t})B(\alpha_{t})+C(\alpha_{t})^{\top}P(\alpha_{t})D(\alpha_{t})\big]\Theta(\alpha_{t})X_{\Theta}(t),X_{\Theta}(t)\Big>\\
&\qquad+\Big<\Theta(\alpha_{t})^{\top}\big[B(\alpha_{t})^{\top}P(\alpha_{t})+D(\alpha_{t})^{\top}P(\alpha_{t})C(\alpha_{t})\big]X_{\Theta}(t),X_{\Theta}(t)\Big>\\
&\qquad+\Big<\Theta(\alpha_{t})^{\top}D(\alpha_{t})^{\top}P(\alpha_{t})D(\alpha_{t})\Theta(\alpha_{t})X_{\Theta}(t),X_{\Theta}(t)\Big>\Big\}dt.
\end{align*}
Taking the limit as \( T \to \infty \) yields:
\begin{equation}\label{step-3-2}
\begin{aligned}
&-\left<P(i)x,x\right>=\mathbb{E}\int_{0}^{\infty}\Big\{\Big<\big[\mathcal{M}(\mathbf{P},\alpha_{t})-Q(\alpha_{t})\big]X_{\Theta}(t),X_{\Theta}(t)\Big>
+\Big<\big[\mathcal{L}(\mathbf{P},\alpha_{t})-S(\alpha_{t})^{\top}\big]\Theta(\alpha_{t})X_{\Theta}(t),X_{\Theta}(t)\Big>\\
&\quad+\Big<\Theta(\alpha_{t})^{\top}\big[\mathcal{L}(\mathbf{P},\alpha_{t})^{\top}-S(\alpha_{t})\big]X_{\Theta}(t),X_{\Theta}(t)\Big>
+\Big<\Theta(\alpha_{t})^{\top}\big[\mathcal{N}(\mathbf{P},\alpha_{t})-R(\alpha_{t})\big]\Theta(\alpha_{t})X_{\Theta}(t),X_{\Theta}(t)\Big>\Big\}dt.
\end{aligned}
\end{equation}
Combining equations \eqref{step-3-1} and \eqref{step-3-2} gives
$$
\begin{aligned}
&\quad V_{\infty}(x,i)-\left<P(i)x,x\right>\\
&=\mathbb{E}\int_{0}^{\infty}\left\langle\left(\mathcal{M}(\mathbf{P},\alpha_{t})+\mathcal{L}(\mathbf{P},\alpha_{t})\Theta(\alpha_{t})+\Theta(\alpha_{t})^{\top}\mathcal{L}(\mathbf{P},\alpha_{t})^{\top}
+\Theta(\alpha_{t})^{\top}\mathcal{N}(\mathbf{P},\alpha_{t})\Theta(\alpha_{t})\right)X_{\Theta}(t),X_{\Theta}(t)\right\rangle d t.
\end{aligned}
$$
From the definition of $\Theta(i)$ and noting that $\mathbf{P}$ solves CAREs \eqref{CAREs-infinite-2}, we can verify that 
$$
\mathcal{M}(\mathbf{P},i)+\mathcal{L}(\mathbf{P},i)\Theta(i)+\Theta(i)^{\top}\mathcal{L}(\mathbf{P},i)^{\top}+\Theta(i)^{\top}\mathcal{N}(\mathbf{P},i)\Theta(i)=0,\quad \forall i\in\mathcal{S},
$$
Hence, we have
$$
V_{\infty}(x,i)-\left<P(i)x,x\right>=0,\quad \forall (x,i)\in\mathbb{R}^{n}\times\mathcal{S},
$$
which implies \eqref{value-infinite} and completes the proof.
\end{proof}

Proposition \ref{prop-CAREs-existence} resolves question \textbf{Q3} posed in Remark \ref{rmk-question}. Combining it with Theorems \ref{thm-open-loop-solvable-infty}, \ref{thm-unique-solvability}, and Proposition \ref{prop-CAREs-existence} immediately yields the following corollary.
\begin{corollary}\label{coro-infinite}
Let assumptions (A1)–(A2) hold. Then the following statements are true:
\begin{description}
    \item[(i)] The CAREs \eqref{CAREs-infinite} admit a unique solution $\mathbf{P_{\infty}}\in\mathcal{D}(\mathbb{S}^{n})$ satisfying the condition \eqref{Theta-infinity};
    \item[(ii)] For any initial pair $(x,i)$, Problem (M-ZLQ)$_{\infty}$ admits a unique open-loop saddle strategy $\bar{u}_{\infty}(\cdot)$ with closed-loop representation given by \eqref{closed-loop-representation-infty};
    \item[(iii)] The value function of Problem (M-ZLQ)$_{\infty}$ is given by
    $$
    V_{\infty}(x,i)=\langle P_{\infty}(i)x,x\rangle,\quad \forall (x,i)\in \mathbb{R}^{n}\times\mathcal{S}.
    $$
    \end{description}
\end{corollary}

From the proof of Proposition \ref{prop-CAREs-existence}, we know that for any $i \in \mathcal{S}$, $\left| P_{T}(t,i) \right| \leq K$ for some constant $K > 0$.  Based on this result, we can further derive the following result.

\begin{corollary}
 Let (A1)-(A2) hold, and $\mathbf{P_{T}}(\cdot)\in\mathcal{D}\left(C\left([0,T],\mathbb{S}^{n}\right)\right)$ be the solution to CDREs \eqref{CDREs-T}.
 Then there exists a constant $K>0$ such that
 $$
 \left|\mathcal{N}(t;\mathbf{P_{T}},i)^{-1}\right|\leq K,\quad 0\leq t\leq T<\infty.
 $$
\end{corollary}
\begin{proof}
  Clearly, by notation \eqref{notation-MLN-infinite}, we have
  $$
\mathcal{N}(t;\mathbf{P},i)= \left(\begin{array}{cc}
\mathcal{N}_{11}(t;\mathbf{P_{T}},i) &\mathcal{N}_{12}(t;\mathbf{P_{T}},i)\\\mathcal{N}_{21}(t;\mathbf{P_{T}},i) &\mathcal{N}_{22}(t;\mathbf{P_{T}},i)
\end{array}\right),\quad \forall i\in\mathcal{S}.
  $$
From \eqref{strongly-regular-condition-T} and \eqref{PT-bound}, we have, for all \(i \in \mathcal{S}\),
\[
\mathcal{N}_{11}(t;\mathbf{P_{T}},i) \geq \delta I, \quad
\mathcal{N}_{22}(t;\mathbf{P_{T}},i) \leq -\delta I, \quad
|\mathcal{N}_{12}(t;\mathbf{P_{T}},i)| = |\mathcal{N}_{21}(t;\mathbf{P_{T}},i)| \leq \rho,
\]
where the positive constants \(\delta,\rho\) are independent of \(t\) and \(T.\)
Moreover, one can easily verify that the inverse of $\mathcal{N}(t;\mathbf{P},i)$ is given by
$$
\left(\begin{matrix}
\widetilde{\mathcal{N}}_{11}(t;\mathbf{P_{T}},i) & \widetilde{\mathcal{N}}_{12}(t;\mathbf{P_{T}},i)\\\widetilde{\mathcal{N}}_{21}(t;\mathbf{P_{T}},i)&\widetilde{\mathcal{N}}_{22}(t;\mathbf{P_{T}},i)
\end{matrix}\right),\quad i\in\mathcal{S},
$$
where
$$\begin{aligned}
&\widetilde{\mathcal{N}}_{11}(t;\mathbf{P_{T}},i)=\mathcal{N}_{11}(t;\mathbf{P_{T}},i)^{-1}+\left(\mathcal{N}_{11}(t;\mathbf{P_{T}},i)^{-1}\mathcal{N}_{12}(t;\mathbf{P_{T}},i)\right)\widetilde{\mathcal{N}}_{22}(t;\mathbf{P_{T}},i)\left(\mathcal{N}_{11}(t;\mathbf{P_{T}},i)^{-1}\mathcal{N}_{12}(t;\mathbf{P_{T}},i)\right)^{\top},\\
&\widetilde{\mathcal{N}}_{12}(t;\mathbf{P_{T}},i)=\widetilde{\mathcal{N}}_{21}(t;\mathbf{P_{T}},i)^{\top}=-\left(\mathcal{N}_{11}(t;\mathbf{P_{T}},i)^{-1}\mathcal{N}_{12}(t;\mathbf{P_{T}},i)\right)\widetilde{\mathcal{N}}_{22}(t;\mathbf{P_{T}},i),\\
&\widetilde{\mathcal{N}}_{22}(t;\mathbf{P_{T}},i)=\left[\mathcal{N}_{22}(t;\mathbf{P_{T}},i)-\mathcal{N}_{21}(t;\mathbf{P_{T}},i)\mathcal{N}_{11}(t;\mathbf{P_{T}},i)^{-1}\mathcal{N}_{12}(t;\mathbf{P_{T}},i)\right]^{-1}.
\end{aligned}$$
Note that for $i\in\mathcal{S}$
$$
\mathcal{N}_{22}(t;\mathbf{P_{T}},i)-\mathcal{N}_{21}(t;\mathbf{P_{T}},i)\mathcal{N}_{11}(t;\mathbf{P_{T}},i)^{-1}\mathcal{N}_{12}(t;\mathbf{P_{T}},i)\leq -\delta I.
$$
Hence, we have
$$
\left|\mathcal{N}_{11}(t;\mathbf{P_{T}},i)^{-1}\right|\leq \sqrt{m_1}\delta^{-1},\quad \left|\widetilde{\mathcal{N}}_{22}(t;\mathbf{P_{T}},i)\right|\leq \sqrt{m_2}\delta^{-1},
$$
which further implies the desired result.
\end{proof}

\section{The turnpike property}\label{section-5}

In this section, we establish the turnpike property \eqref{exponential-turnpike-property-introduction}. Before proceeding, we provide some preparatory results. Theorem \ref{thm-convergence} shows that the solution $\mathbf{P_{T}}(\cdot)$ to CDREs \eqref{CDREs-T} converges to a unique solution $\mathbf{P}$ of CAREs \eqref{CAREs-infinite-2}. 
The following result quantifies the convergence rate.

\begin{theorem}\label{thm-convergence-rate}
Assume (A1)-(A2) hold. Let $\mathbf{P}_T(\cdot)$ be the unique strongly regular solution to CDREs \eqref{CDREs-T}, and let $\mathbf{P}$ be defined by \eqref{P-limit-2}. Then there exist constants $K,\,\mu > 0$, independent of $T$, such that
\begin{equation}\label{convergence-rate}
  \left|P(i)-P_{T}(t,i)\right|\leq Ke^{-\mu(T-t)},\quad \forall i\in\mathcal{S}.
\end{equation}
\end{theorem}

\begin{proof}
  Let
  $$\Sigma_{T}(t,i)\triangleq P(i)-P_{T}(t,i),\quad i\in\mathcal{S},\quad \text{and}\quad \left|\Sigma_{T}(t)\right|\triangleq \max_{i\in\mathcal{S}}\left|\Sigma_{T}(t,i)\right|.$$
  Then we only need to verify
  \begin{equation}\label{rate-0}
  \left|\Sigma_{T}(t)\right|\leq Ke^{-\mu(T-t)}.
  \end{equation}
  By some direct computations, we have
  \begin{equation}\label{rate-1}
  \begin{aligned}
\dot{\Sigma}_{T}(t,i)&=-\Sigma_{T}(t,i)\left[A(i)+B(i)\Theta(i)\right]-\left[A(i)+B(i)\Theta(i)\right]^{\top}\Sigma_{T}(t,i)-\sum_{j=1}^{L}\pi_{ij}\Sigma_{T}(t,j)\\
&\quad -\left[C(i)+D(i)\Theta(i)\right]^{\top}\Sigma_{T}(t,i)\left[C(i)+D(i)\Theta(i)\right]\\
&\quad -\left[\Theta(i)-\Theta_{T}(t,i)\right]^{\top}\mathcal{N}(t;\mathbf{P_{T}},i)\left[\Theta(i)-\Theta_{T}(t,i)\right],
   \end{aligned}
   \end{equation}
  where
  \begin{align*}
  \Theta(i)=-\mathcal{N}(\mathbf{P},i)^{-1}\mathcal{L}(\mathbf{P},i)^{\top},\quad
  \Theta_{T}(t,i)=-\mathcal{N}(t;\mathbf{P_{T}},i)^{-1}\mathcal{L}(t;\mathbf{P_{T}},i)^{\top}.
\end{align*}

Let $\Phi(\cdot;t,i)$ be the solution to SDE
\begin{equation}\label{rate-2}
\left\{
\begin{aligned}
&d\Phi(s;t,i)=\left[A(\alpha_{s})+B(\alpha_{s})\Theta(\alpha_{s})\right]\Phi(s;t,i)ds+\left[C(\alpha_{s})+D(\alpha_{s})\Theta(\alpha_{s})\right]\Phi(s;t,i)dW(s)\\
&\Phi(t;t,i)=I_{n},\quad \alpha_{t}=i.
\end{aligned}
\right.
\end{equation}
Applying It\^o's rule to $\left<\Sigma(s,\alpha_{s})\Phi(s;t,i),\Phi(s;t,i)\right>$ yields 
\begin{equation}\label{rate-3}
\begin{aligned}
  \Sigma_{T}(t,i)&=\mathbb{E}\Big\{\int_{t}^{k}\big<\left[\Theta(\alpha_{s})-\Theta_{T}(s,\alpha_{s})\right]^{\top}\mathcal{N}(s;\mathbf{P_{T}},\alpha_{s})\left[\Theta(\alpha_{s})-\Theta_{T}(s,\alpha_{s})\right]\Phi(s;t,i),\Phi(s;t,i)\big>ds\\
  &\quad +\big<\Sigma_{T}(k,\alpha_{k})\Phi(k;t,i),\Phi(k;t,i)\big>\Big\},\quad 0\leq t\leq k\leq T,\quad \forall i\in\mathcal{S}.
  \end{aligned}
\end{equation}
From
\begin{equation}\label{rate-Theta}
\Theta(i)-\Theta_{T}(t,i)=-\mathcal{N}(t;\mathbf{P_{T}},i)^{-1}\left[B(i)^{\top}\Sigma_{T}(t,i)+D(i)^{\top}\Sigma_{T}(t,i)\left(C(i)+D(i)\Theta(i)\right)\right],\quad i\in\mathcal{S},
\end{equation}
and the bound
$$
\left|\mathcal{N}(t;\mathbf{P_{T}},i)^{-1}\right|\leq K,\quad \forall 0\leq t\leq T<\infty,\quad \forall i\in\mathcal{S},\quad \text{for some }K>0,
$$
it follows that 
\begin{equation}\label{rate-4}
\left|\left[\Theta(\alpha_{t})-\Theta_{T}(t,\alpha_{t})\right]^{\top}\mathcal{N}(t;\mathbf{P_{T}},\alpha_{t})\left[\Theta(\alpha_{t})-\Theta_{T}(t,\alpha_{t})\right]\right|\leq K_{1}\left|\Sigma_{T}(t)\right|,\quad \text{for some }K_{1}>0.
\end{equation}
By 
Proposition \ref{prop-CAREs-existence}, the system $[A+B\Theta,C+D\Theta]_{\alpha}$ is $L^{2}$-stable. Therefore, by Corollary \ref{coro-AC}, there exist a constant $K_{2},\,\,\mu>0$ such that
\begin{equation}\label{rate-5}
\mathbb{E}\left|\Phi(s;t,i)\right|^{2}<K_{2}e^{-\mu(s-t)}.
\end{equation}
Moreover, noting that $\lim_{T\to \infty}\Sigma_T(0,i)=0$. Therefore, 
we can choose an integer $N > 0$ such that
$$
K_{2}\left|\Sigma_{T}(0)\right|\leq \rho\triangleq \frac{\mu}{2K_{1}K_{2}},\quad \forall T\geq N.
$$

Now, we begin to verify \eqref{rate-0} holds. For the case $0\leq t\leq T\leq 2N$, the relation \eqref{relation} implies that
$$
\left|\Sigma_{T}(t)\right|=\left|\Sigma_{T-t}(0)\right|\leq  \max_{s\in[0,2N]}\left|\Sigma_{s}(0)\right|\triangleq K_{3}\leq \left(K_{3}e^{2N\mu}\right)e^{-\mu(T-t)}.
$$
For the case $T>2N$, let $k\geq N$ be the integer such that
$$
N+k\leq T<N+k+1.
$$
Then we divide the proof of \eqref{rate-0}  into the following steps.

Step 1, for $t\in[k,T]$ and $T>2N$, we have $0\leq T-t\leq T-k\leq N+1$, which further yields
$$
\left|\Sigma_{T}(t)\right|=\left|\Sigma_{T-t}(0)\right|\leq  \max_{s\in[0,N+1]}\left|\Sigma_{s}(0)\right|\triangleq K_{4}\leq \left(K_{4}e^{\mu(N+1)}\right)e^{-\mu(T-t)}.
$$

Step 2, for $t\in[0,k]$ and $T>2N$, it follows from \eqref{rate-3}-\eqref{rate-5} that
\begin{equation}\label{rate-6}
  \begin{aligned}
 \left| \Sigma_{T}(t)\right|&\leq \left|\Sigma_{T-k}(0)\right|\mathbb{E}\left|\Phi(k;t,i)\right|^{2}+K_{1}\int_{t}^{k}\mathbb{E}\left| \Sigma_{T}(s)\right|^{2}\left|\Phi(s;t,i)\right|^{2}ds\\
 &\leq \rho e^{-\mu(k-t)}+K_{1}K_{2}\int_{t}^{k}e^{-\mu(s-t)}\left| \Sigma_{T}(s)\right|^{2}ds.
  \end{aligned}
\end{equation}
Let $K_{5}=K_{1}K_{2}$ and
\begin{equation}\label{rate-7}
h(t)\triangleq K_{5}e^{\mu t}\left|\Sigma_{T}(k-t)\right|,\quad 0\leq t\leq k.
\end{equation}
Then equation \eqref{rate-6} yields
$$
h(t)\leq K_{5}\rho+\int_{0}^{t}e^{-\mu s}h(s)^{2}ds=\frac{\mu}{2}+\int_{0}^{t}e^{-\mu s}h(s)^{2}ds,\quad 0\leq t\leq k.
$$
Now, set
$$
H(t)\triangleq \int_{0}^{t}e^{-\mu s}h(s)^{2}ds,\quad 0\leq t\leq k.
$$
Then, $H(0)=0$ and $h(t)\leq \dfrac{\mu}{2}+H(t)$, which further implies
$$
e^{\mu t}H^{\prime}(t)=h(t)^{2}\leq \left[\frac{\mu}{2}+H(t)\right]^{2},\quad 0\leq t\leq k.
$$
Consequently, we have
$$
d\left[-\frac{1}{\frac{\mu}{2}+H(t)}\right]=\frac{H^{\prime}(t)}{\left[\frac{\mu}{2}+H(t)\right]^{2}}\leq e^{-\mu t},\quad 0\leq t\leq k.
$$
Integrating both sides yields
$$
\frac{2}{\mu}-\frac{1}{\frac{\mu}{2}+H(t)}\leq \int_{0}^{t}e^{-\mu s}ds\leq \dfrac{1}{\mu},\quad 0\leq t\leq k,
$$
which implies $H(t)\leq \dfrac{\mu}{2}$ and hence $h(t)\leq \mu$ for $t\in[0,k]$. Consequently, for $t\in[0,k]$, by \eqref{rate-7}, we have
$$
\left|\Sigma_{T}(t)\right|=\frac{1}{K_{5}}e^{-\mu(k-t)}h(k-t)\leq\frac{\mu}{K_{5}}e^{-\mu(k-t)}=\frac{\mu}{K_{5}}e^{\mu(T-k)}e^{-\mu(T-t)}\leq \frac{\mu}{K_{5}}e^{\mu(N+1)}e^{-\mu(T-t)}.
$$
To sum up, we have completed the proof.
\end{proof}

By Theorem \ref{thm-convergence}, the matrix $\mathbf{P}$ defined in \eqref{P-limit-2} is the unique solution to CAREs \eqref{CAREs-infinite}. 
Therefore, result \eqref{convergence-rate} can be restated as:
\begin{equation}\label{convergence-rate-2}
  \left|P_{\infty}(i)-P_{T}(t,i)\right|\leq Ke^{-\mu(T-t)},\quad \forall i\in\mathcal{S},
\end{equation}
and  equation \eqref{rate-Theta} further implies that
\begin{equation}\label{rate-Theta-2}
\left|\Theta_{\infty}(i)-\Theta_{T}(t,i)\right|\leq Ke^{-\mu(T-t)},\quad \forall i\in\mathcal{S},
\end{equation}
 for some $K,\,\,\mu>0$, independent of $T$.

Based on the above result, we also present the turnpike property for Problem (M-ZLQ)$_{T}$.
\begin{theorem}\label{thm-turnpike-property}
 Let (A1)-(A2) hold. Then for any initial pair $(x,i)$, there exist constants $K > 0$ and $\mu > 0$, independent of $T$, such that
 \begin{equation}\label{exponential-turnpike-property}
\mathbb{E}\left[
\left|\bar{X}_{T}(t)-\bar{X}_{\infty}(t)\right|^{2}+\left|\bar{u}_{1,T}(t)-\bar{u}_{1,\infty}(t)\right|^{2}+\left|\bar{u}_{2,T}(t)-\bar{u}_{2,\infty}(t)\right|^{2}
\right]\leq K\left|x\right|^{2}\left[e^{-\mu(T-t)}+e^{-\mu t}\right],
\end{equation}
where $(\bar{X}_{T}(\cdot),\bar{u}_{1,T}(\cdot),\bar{u}_{2,T}(\cdot))$ (respectively, $(\bar{X}_{\infty}(\cdot),\bar{u}_{1,\infty}(\cdot),\bar{u}_{2,\infty}(\cdot))$) is the corresponding optimal triple of Problem (M-ZLQ)$_{T}$ (respectively, Problem (M-ZLQ)$_{\infty}$).
\end{theorem}

\begin{proof}
Let $\mathcal{X}_{T}(\cdot)\triangleq \bar{X}_{T}(\cdot)-\bar{X}_{\infty}(\cdot)$,
 $
 \bar{u}_{T}(\cdot)=\left(\begin{array}{c}\bar{u}_{1,T}(\cdot)\\ \bar{u}_{2,T}(\cdot)\end{array}\right),\   \bar{u}_{\infty}(\cdot)=\left(\begin{array}{c}\bar{u}_{1,\infty}(\cdot)\\ \bar{u}_{2,\infty}(\cdot)\end{array}\right),
 $
and set
\begin{align*}
 & A_{\Theta_{\infty}}(i)=A(i)+B(i)\Theta_{\infty}(i),\quad C_{\Theta_{\infty}}(i)=C(i)+D(i)\Theta_{\infty}(i),\\
 & A_{\Theta_{T}}(t,i)=A(i)+B(i)\Theta_{T}(t,i),\quad C_{\Theta_{T}}(t,i)=C(i)+D(i)\Theta_{T}(t,i),
\end{align*}
where $\mathbf{\Theta}$ and $\mathbf{\Theta_{T}}(\cdot)$ are defined in \eqref{Theta-T} and \eqref{Theta-infinity}, respectively.
From \eqref{closed-loop-state-T} and \eqref{closed-loop-state-infinite}, one has
\begin{equation}\label{closed-loop-state-X}
   \left\{
   \begin{aligned}
   &d\mathcal{X}_{T}(t)=\left[A_{\Theta_{T}}(t,\alpha_{t})\bar{X}_{T}(t)-A_{\Theta_{\infty}}(\alpha_{t})\bar{X}_{\infty}(t)\right]dt
   +\left[C_{\Theta_{T}}(t,\alpha_{t})\bar{X}_{T}(t)-C_{\Theta_{\infty}}(\alpha_{t})\bar{X}_{\infty}(t)\right]dW(t)\\
   &\mathcal{X}_{T}(0)=0,\quad \alpha_{0}=i,
   \end{aligned}
   \right.
  \end{equation}
  Note that the system $[A_{\Theta_{\infty}},C_{\Theta_{\infty}}]_{\alpha}$ is $L^{2}$-stable. Then by Proposition 2.2 in \cite{My-paper-2025-infinite-M-SLQ}, there exist a $\mathbf{\Sigma}\in \mathcal{D}\left(\mathbb{S}^{n}_+\right)$ and $\mu_{1}>0$ such that
  $$
  \Sigma(i)A_{\Theta_{\infty}}(i)+A_{\Theta_{\infty}}(i)^{\top}\Sigma(i)+C_{\Theta_{\infty}}(i)^{\top}\Sigma(i)C_{\Theta_{\infty}}(i)+\sum_{j=1}^{L}\pi_{ij}\Sigma(j)\leq -2\mu_{1}\Sigma(i),\quad i\in\mathcal{S}.
  $$
  Additionally, by \eqref{rate-Theta-2} and Corollary \ref{coro-AC}, one has
  \begin{align*}
      \left|\Theta_{\infty}-\Theta_{T}(t)\right|\triangleq \max_{i\in\mathcal{S}}\left|\Theta_{\infty}(i)-\Theta_{T}(t,i)\right|\leq K e^{-\mu_{2}(T-t)},\quad
      \mathbb{E} \left|X_{T}(t)\right|^{2}\leq K\left|x\right|^{2} e^{-\mu_{3}t},
  \end{align*}
  for some $\mu_{2},\,\,\mu_{3}>0$.
  Let $\mu\triangleq\min\{\mu_{1},\mu_{2},\mu_{3}\}$ and $K>0$ be a constant that may assume different values in each row. Applying It\^o's rule to $\left<\Sigma(\alpha_{t})\mathcal{X}_{T}(t),\mathcal{X}_{T}(t)\right>$ yields
  \begin{align*}
  &\quad \frac{d}{dt}\mathbb{E}\left|\Sigma(\alpha_{t})^{\frac{1}{2}}\mathcal{X}_{T}(t)\right|^{2}=\frac{d}{dt}\mathbb{E}\big<\Sigma(\alpha_{t})\mathcal{X}_{T}(t),\mathcal{X}_{T}(t)\big>\\
  &=\mathbb{E}\Big[\big<\Sigma(\alpha_{t})\left(A_{\Theta_{T}}(t,\alpha_{t})\bar{X}_{T}(t)-A_{\Theta_{\infty}}(\alpha_{t})\bar{X}_{\infty}(t)\right),\mathcal{X}_{T}(t)\big>
  +\big<\sum_{j=1}^{L}\pi_{\alpha_{t}j}\Sigma(j)\mathcal{X}_{T}(t),\mathcal{X}_{T}(t)\big>\\
  &\quad+\big<\Sigma(\alpha_{t})\left(C_{\Theta_{T}}(t,\alpha_{t})\bar{X}_{T}(t)-C_{\Theta_{\infty}}(\alpha_{t})\bar{X}_{\infty}(t)\right),C_{\Theta_{T}}(t,\alpha_{t})\bar{X}_{T}(t)-C_{\Theta_{\infty}}(\alpha_{t})\bar{X}_{\infty}(t)\big>\\
    &\quad+\big<\Sigma(\alpha_{t})\mathcal{X}_{T}(t),A_{\Theta_{T}}(t,\alpha_{t})\bar{X}_{T}(t)-A_{\Theta_{\infty}}(\alpha_{t})\bar{X}_{\infty}(t)\big>\Big]\\
    &=\mathbb{E}\Big[\big<\big(\Sigma(\alpha_{t})A_{\Theta_{\infty}}(\alpha_{t})+A_{\Theta_{\infty}}(\alpha_{t})^{\top}\Sigma(\alpha_{t})+C_{\Theta_{\infty}}(\alpha_{t})^{\top}\Sigma(\alpha_{t})C_{\Theta_{\infty}}(\alpha_{t})
    +\sum_{j=1}^{L}\pi_{\alpha_{t}j}\Sigma(j)\big)\mathcal{X}_{T}(t),\mathcal{X}_{T}(t)\big>\\
  &\quad+\big<\Sigma(\alpha_{t})\big(A_{\Theta_{T}}(t,\alpha_{t})-A_{\Theta_{\infty}}(\alpha_{t})\big)\bar{X}_{T}(t),\mathcal{X}_{T}(t)\big>
  +\big<\Sigma(\alpha_{t})\mathcal{X}_{T}(t),\big(A_{\Theta_{T}}(t,\alpha_{t})-A_{\Theta_{\infty}}(\alpha_{t})\big)\bar{X}_{T}(t)\big>\\
  &\quad+\big<\Sigma(\alpha_{t})\big(C_{\Theta_{T}}(t,\alpha_{t})-C_{\Theta_{\infty}}(\alpha_{t})\big)\bar{X}_{T}(t),\big(C_{\Theta_{T}}(t,\alpha_{t})-C_{\Theta_{\infty}}(\alpha_{t})\big)\bar{X}_{T}(t)\big>\\
    &\quad+2\big<\Sigma(\alpha_{t})\big(C_{\Theta_{T}}(t,\alpha_{t})-C_{\Theta_{\infty}}(\alpha_{t})\big)\bar{X}_{T}(t),C_{\Theta_{\infty}}(\alpha_{t})\mathcal{X}_{T}(t)\big>\Big]\\
    &\leq -2\mu_{1}\mathbb{E}|\Sigma(\alpha_{t})^{\frac{1}{2}}\mathcal{X}_{T}(t)|^{2}
    +K|\Theta_{\infty}-\Theta_{T}(t)|\mathbb{E}\left[|\Sigma(\alpha_{t})^{\frac{1}{2}}\mathcal{X}_{T}(t)||\bar{X}_{T}(t)|\right]\\
    &\quad +K|\Theta_{\infty}-\Theta_{T}(t)|^{2}\mathbb{E}|\bar{X}_{T}(t)|^{2}\\
    &\leq -\mu_{1}\mathbb{E}|\Sigma(\alpha_{t})^{\frac{1}{2}}\mathcal{X}_{T}(t)|^{2}+K|x|^{2}e^{-2\mu_{2}(T-t)}e^{-\mu_{3} t}\\
    &\leq -\mu\mathbb{E}|\Sigma(\alpha_{t})^{\frac{1}{2}}\mathcal{X}_{T}(t)|^{2}+K|x|^{2}e^{-2\mu(T-t)}e^{-\mu t}\\
    &\leq -\mu\mathbb{E}|\Sigma(\alpha_{t})^{\frac{1}{2}}\mathcal{X}_{T}(t)|^{2}+K|x|^{2}e^{-\mu(T-t)}e^{-\mu t}.
  \end{align*}
By Gronwall's inequality, it follows that:
\begin{align*}
 \mathbb{E}\left|\Sigma(\alpha_{t})^{\frac{1}{2}}\mathcal{X}_{T}(t)\right|^{2} & \leq K\left|x\right|^{2}\int_{0}^{t}e^{-\mu(t-r)}e^{-\mu(T-r)}e^{-\mu r}dr=K\left|x\right|^{2}e^{-\mu t}\int_{0}^{t}e^{-\mu(T-r)}dr\\
 &\leq K\left|x\right|^{2}e^{-\mu t} e^{-\mu (T-t)}\leq K\left|x\right|^{2}e^{-\mu t/2} e^{-\mu (T-t)/2}\\
 &\leq K\left|x\right|^{2}(e^{-\mu t} + e^{-\mu (T-t)}).
\end{align*}

  Consequently, we have
  $$
 \mathbb{E}\left|\bar{X}_{T}(t)-\bar{X}_{\infty}(t)\right|^{2}= \mathbb{E}\left|\mathcal{X}_{T}(t)\right|^{2}\leq K\left|x\right|^{2}\left[e^{-\mu(T-t)}+e^{-\mu t}\right].
  $$

  On the other hand, by Corollary \ref{coro-T} and Corollary \ref{coro-infinite}, we have
\begin{align*}
 &\bar{u}_{T}(t)=\Theta_{T}(t,\alpha_{t})\bar{X}_{T}(t),\quad t\in[0,T],\\
 &\bar{u}_{\infty}(t)=\Theta_{\infty}(\alpha_{t})\bar{X}_{\infty}(t),\quad t\in[0,\infty),
\end{align*}
which implies
\begin{align*}
 \mathbb{E}\left|\bar{u}_{T}(t)-\bar{u}_{\infty}(t)\right|^{2}&= \mathbb{E}\left|\Theta_{T}(t,\alpha_{t})\bar{X}_{T}(t)-\Theta_{\infty}(\alpha_{t})\bar{X}_{\infty}(t)\right|^{2}\\
 &\leq 2\mathbb{E}\left[\left|\Theta_{T}(t,\alpha_{t})-\Theta_{\infty}(\alpha_{t})\right|^{2}\left|\bar{X}_{T}(t)\right|^{2}\right]+2\mathbb{E}\left[\left|\Theta_{\infty}(\alpha_{t})\right|^{2}\left|\mathcal{X}_{T}(t)\right|^{2}\right]\\
 &\leq K\left|x\right|^{2}\left[e^{-\mu(T-t)}+e^{-\mu t}\right].
\end{align*}
To sum up, we have completed the proof.
\end{proof}


\end{document}